\documentclass[11pt]{amsart} \textwidth=14.5cm \oddsidemargin=1cm
\usepackage{amsmath,wasysym}
\usepackage{amsxtra}
\usepackage{amscd}
\usepackage{amsthm}
\usepackage{amsfonts}
\usepackage{amssymb}
\usepackage{eucal}
\usepackage{graphics, color}
\usepackage{hyperref,mathrsfs}
\usepackage[usenames,dvipsnames]{xcolor}
\usepackage{tikz}
\usepackage{stmaryrd}
\usepackage[all]{xy}
\usepackage{hyperref}
\usepackage{color}
\usepackage{bbm} 
\usepackage{tikz}
\usetikzlibrary{arrows.meta}
\usetikzlibrary{arrows.meta,positioning}
\allowdisplaybreaks

\hypersetup{colorlinks,linkcolor=blue, urlcolor=blue,
citecolor=blue}

\newtheorem{thm}{Theorem} [section]
\newtheorem{lem}[thm]{Lemma}

\newtheorem{prop}[thm]{Proposition}

\theoremstyle{definition}
\newtheorem{definition}[thm]{Definition}

\newtheorem{rem}[thm]{Remark}

\numberwithin{equation}{section}

\newcommand{\C}{{\mathbb C}}

\newcommand{\End}{{\mathrm{End}}}
\newcommand{\F}{\mathbb{F}}

\newcommand{\h}{{\mathfrak h}}

\newcommand{\Z}{{\mathbb Z}}

\newcommand{\nc}{\newcommand}
\nc{\browntext}[1]{\textcolor{brown}{#1}}
\nc{\greentext}[1]{\textcolor{green}{#1}}
\nc{\redtext}[1]{\textcolor{red}{#1}}
\nc{\bluetext}[1]{\textcolor{blue}{#1}}
\nc{\brown}[1]{\browntext{ #1}}
\nc{\green}[1]{\greentext{ #1}}
\nc{\red}[1]{\redtext{ #1}}
\nc{\blue}[1]{\bluetext{ #1}}

\title[Schur--Weyl duality for coideal subalgebras of quantum affine superalgebras of type $A$]{Schur--Weyl duality for coideal subalgebras of quantum affine superalgebras of type $A$}

\author{Weideng Cui}
\address{School of Mathematics, Shandong University, Jinan, Shandong 250100, China}
\email{cwdeng@amss.ac.cn}


\begin{document}

\begin{abstract}
We extend the results of \cite{FLLLWW20} to the super setting. More precisely, we define two families of coideal sub-superalgebras of quantum affine superalgebras of type $A$, called $\jmath$- and $\imath$-quantum affine superalgebras respectively, and establish Schur--Weyl duality between them and affine Hecke algebras of type $C$.
\end{abstract}

\maketitle
\setcounter{tocdepth}{1}
\tableofcontents

\section{Introduction}
Schur--Weyl duality is a fundamental theme in representation theory. In its classical form, it describes the double centralizer relationship between the actions of a general linear group and a symmetric group on some tensor space. Its quantum counterpart replaces these algebras with quantized enveloping algebras and Hecke algebras of type $A$; cf. \cite{J86}. Affine versions of this correspondence provide an important connection between quantum affine algebras of type $A$, affine Hecke algebras of type $A$, and affine $q$-Schur algebras; see [VV99, DDF12]. Analogous duality has also been established for orthogonal groups or symplectic groups and Brauer algebras; cf. \cite{Br37}. An essential feature of these dualities is that the tensor space carries the commuting actions of two algebras, and under suitable hypotheses, each algebra forms the full centralizer of the other.

Representation theory of Lie superalgebras was strongly motivated by applications in physics and in topology. In recent years Schur--Weyl duality and categorification have shed light on the interaction of representation theory of Lie superalgebras with other branches of mathematics; we refer to \cite{Se14} for a survey. The quantum deformations of universal enveloping algebras of simple Lie superalgebras, together with their affine versions, have been defined in terms of generators and relations by Yamane; cf. \cite{Yam94} and \cite{Yam99}. They have been studied from various viewpoints; cf. \cite{BKK00}, \cite{Z14}, among others.

An important development in this direction is the study of Schur dualities associated with Lie superalgebras. Many classical Schur dualities admit super analogues. For example, Schur's work has been extended to general linear Lie superalgebras in \cite{Ser84} and \cite{BR87}, Jimbo's work has been generalized to quantized enveloping superalgebras in \cite{Mo03} and \cite{Mi06}, Brauer's work has been expanded to ortho-symplectic Lie superalgebras in \cite{ES16}.

Another important development in this direction is the study of Schur dualities associated with quantum symmetric pairs; cf. \cite{W23}. In this setting, the role of quantized enveloping algebras is played by a distinguished coideal subalgebra, commonly called an $\imath$-quantum group. In particular, $\imath$-quantum groups of type AIII give rise to analogues of the usual Schur--Weyl duality involving Hecke algebras of type $B/C$, called $\imath$Schur dualities; cf. \cite{BW18} and \cite{BWW18}. Affine $\imath$Schur dualities, between coideal subalgebras of quantum affine algebras of type $A$ and affine Hecke algebras of type $C$, were systematically developed in \cite{FLLLW20} and \cite{FLLLW23} by a geometric approach and an algebraic approach, respectively. A three-parameter version was established in \cite{FLLLWW20}. Besides, the author and Shen \cite{CS24} have formulated $\imath$Schur dualities between the $q$-Brauer algebra and the $\imath$-quantum groups of type AI and AII, respectively.

Building on these developments, Shen \cite{S25} has generalized the results of \cite{BWW18} to the super setting; that is, he established an $\imath$Schur duality between an $\imath$-quantum supergroup associated to a super Satake diagram with black nodes and a multi-parameter Hecke algebra of type $B/C$; see also \cite{CL23} for the case of no black nodes. Later on, Shen and Wang \cite{SW25} have defined $\imath$-quantum supergroups associated to quantum supersymmetric pairs of basic types, and formulated an $\imath$Schur duality between a distinguished class of $\imath$-quantum supergroups and $q$-Brauer algebras.

The purpose of this paper is to carry out this strategy further. More precisely, we shall extend the three-parameter affine type $C$ $\imath$Schur duality of \cite{FLLLWW20} to the super setting. The super case is not obtained merely by changing the underlying vector space: the $\mathbb Z_2$-grading affects the defining relations of the quantum affine superalgebra, the tensor product action through the Koszul sign rule, and the affine Hecke algebra action on tensor superspaces. In particular, isotropic odd simple roots give rise to additional quartic Serre relations, while the Hecke algebra action must account for the Koszul signs associated with permutations of homogeneous tensor factors.

We describe the main results in more detail. Based on our previous work \cite{CS26}, we first define two families of coideal sub-superalgebras of quantum affine superalgebras of type $A$, called $\jmath$- and $\imath$-quantum affine superalgebras, respectively, and establish their explicit presentations in terms of generators and relations; cf. Theorems \ref{def:ipreseab} and \ref{def:ipreseabslmadada}. Next we construct an action on an infinite-dimensional tensor superspace of the affine Hecke algebra of type $C$ with three parameters, and prove the commutativity theorem between it and a distinguished class of $\jmath/\imath$-quantum affine superalgebras under suitable specialization of their parameters; cf. Theorems \ref{hatcuth} and \ref{icommu}. We further introduce the corresponding three-parameter affine $\jmath$- and $\imath$-Schur superalgebras as endomorphism algebras of direct sums of signed permutation modules over the affine Hecke algebra of type $C$. By identifying the direct sums of these modules with the relevant tensor superspaces, we obtain algebra homomorphisms from $\jmath/\imath$-quantum affine superalgebras to affine $\jmath/\imath$-Schur superalgebras. Finally, we prove that these homomorphisms are surjective and establish the desired double centralizer properties; cf. Theorems \ref{schruimo} and \ref{suchuri}.

Although most of the ideas follow from \cite{FLLLWW20}, we expect these two families of coideal sub-superalgebras to provide a framework for further investigations. Moreover, several directions for further research arise naturally: (1) exhibit a Drinfeld-type presentation of these superalgebras and develop their representation theory; (2) define and study the coideal sub-superalgebras of quantum affine superalgebras of other types, in particular ortho-symplectic types (cf. \cite{MS26}), and moreover, generalize the results of \cite{SW25} to the affine types; (3) study the two classes of affine $\jmath/\imath$-Schur superalgebras. We shall focus on these topics in future work.


The paper is organized as follows. In Section 2, we recall the relevant quantum affine superalgebras, construct the first family of $\jmath$-quantum affine superalgebras, and establish their presentations and coideal properties. In Section 3, we define the affine type $C$ Hecke algebra action on a tensor superspace and prove the commutativity theorem. In Section 4, we introduce the affine $\jmath$-Schur superalgebras and prove the Schur--Weyl duality theorem. In Section 5, we develop the odd-rank counterpart and establish the second Schur--Weyl duality.

\noindent {\bf The use of AI.} The paper contains numerous calculations; we have verified their validity by hand, and also used ChatGPT 5.6 to assist in reviewing intermediate algebraic manipulations.







\section{The $\jmath$-quantum affine superalgebras}\label{def:quanaffis}
In this section, we introduce a coideal sub-superalgebra of the quantum affine superalgebra of type $A$, called a $\jmath$-quantum affine superalgebra, and moreover, give its explicit presentation in terms of generators and relations.

Throughout the paper, we work over the field $\F=\C(q, Q_{0}, Q_{1})$, where $q, Q_{0}, Q_{1}$ are three indeterminates. When working with superspaces, we denote the parity of a homogeneous element $v$ by $|v| \in \Z_2$; when we write expressions involving $|v|$, we implicitly assume that $v$ is homogeneous. For an element $a\in \Z_2$, we define
\[
    (-1)^{a}
    =
    \begin{cases}
        1 & \text{if $a=\overline{0}$}, \\
        -1 & \text{if $a=\overline{1}$}.
    \end{cases}
\]
For a statement $P$, we define
\[
    \delta_P
    :=
    \begin{cases}
        1 & \text{if $P$ is true}, \\
        0 & \text{otherwise}.
    \end{cases}
\]
In particular, we have the usual Kronecker delta $\delta_{ij} := \delta_{i=j}$.

We fix $r\in \mathbb{Z}_{\geq 2}$, and let $$N=2r+2.$$

Let $V = V_{\overline{0}} \oplus V_{\overline{1}}$ be a vector superspace over $\C$ such that $\dim V_{\overline{0}} =m'$ and $\dim V_{\overline{1}} = n'$ with $m'\neq n'$ and $m'+n'=N$.  We fix a homogeneous basis of $V$:
\[
    v_i,\qquad i \in I_{V} := \left\{1,2,\ldots,N \right\}.
\]
The superspace $\End(V)$, equipped with the supercommutator, forms a Lie superalgebra. We call it the general linear Lie superalgebra and denote it by $\mathfrak{g} = \mathfrak{gl}(V) = \mathfrak{gl}(m'|n',\C)$. The Lie superalgebra $\mathfrak{g}$ has a homogeneous basis given by the matrix units $E_{ij}$, $i,j \in I_{V}$, where the parity of $E_{ij}$ is
\[
    |E_{ij}| = |v_i| + |v_j|.
\]
Let $\mathfrak{h}$ be the Cartan subalgebra of $\mathfrak{g}$ consisting of diagonal matrices. Then the dual space $\mathfrak{h}^{*}$ has a basis $\epsilon_i$, $i \in I_{V}$, defined by $\epsilon_i(E_{jj}) = \delta_{ij}$ for $j \in I_{V}$. We define a symmetric bilinear form on the weight lattice $P = \oplus_{i \in I_{V}} \Z\epsilon_i$ by
\[
    (\epsilon_i, \epsilon_j) =
    (-1)^{|v_i|}\delta_{ij}.
\]
We introduce the notation $|\epsilon_i|$ by setting it to be $|v_i|$ for any $i\in I_{V}$.


We set $I := \left\{ 1,2,\ldots, N-1 \right\}$ and let $\Pi= \left\{ \alpha_i= \epsilon_{i}-\epsilon_{i+1} \mid i \in I \right\}$ be the set of simple roots. Let $P_{\mathrm{af}}=\Z\epsilon_1\oplus \cdots \oplus\Z\epsilon_{N}\oplus \Z \delta$ be a free abelian group of rank $N+1$. We set $\alpha_{0} :=\delta+\epsilon_{N}-\epsilon_{1}$. The symmetric bilinear form on $P$ inducing a symmetric bilinear one on $P_{\mathrm{af}}$ by setting $(\delta, \delta)=(\delta, \epsilon_i)=0$, $\forall i\in I_{V}$, and extending it linearly. For any $i\in I$, we define the parity of $\alpha_i$ by $|\alpha_i|=|\epsilon_i|+|\epsilon_{i+1}|$.

Let $P^\vee = \oplus_{i\in I_{V}} \Z \epsilon^\vee_i$ be the coweight lattice. We have a $\Z$-bilinear pairing
\begin{equation}\label{bilinear form}
    \langle \cdot,\cdot \rangle \colon P^\vee \times P\to \Z, \quad
    \langle \epsilon^\vee_i,\epsilon_j \rangle = \delta_{ij}.
\end{equation}
The set of simple coroots is given by
\begin{equation} \label{hdef}
    \Pi^\vee = \{ h_i  \mid i \in I \},\quad \text{where} ~
	h_i := \epsilon_{i}^\vee - (-1)^{|\alpha_i|} \epsilon_{i+1}^\vee.
\end{equation}
We extend \eqref{bilinear form} to a $\Z$-bilinear pairing $\langle \cdot,\cdot \rangle$ between $P^\vee$ and $P_{\mathrm{af}}$ such that $\langle \epsilon^\vee_i, \delta \rangle=0$ for any $i \in I_{V}$. Set $\epsilon_{0}^\vee :=\epsilon_{N}^\vee$ and $h_{0} := \epsilon_{0}^\vee - (-1)^{|\alpha_{0}|} \epsilon_{1}^\vee$, where $|\alpha_{0}|=|\epsilon_{N}|+|\epsilon_{1}|$. We define the affine Cartan matrix $A^{\mathrm{af}}=(a_{ij})_{i, j\in \hat{I}}$ associated to the above data by
\begin{equation} \label{car matrix}
    a_{ij} = \langle h_i, \alpha_j \rangle
    = \left( 1 + (-1)^{|\alpha_i|} \right) \delta_{ij} - \delta_{i,j+1} - (-1)^{|\alpha_i|} \delta_{i,j-1},\quad \forall i, j\in \hat{I},
\end{equation}
where $\hat{I}=\left\{ 0,1,2,\ldots, N-1 \right\}$. Set $d_i := (-1)^{|\epsilon_i|}$ for any $i\in I_{V}$ and $d_{0} :=d_{N}$. Then we have
\begin{equation} \label{car matrix diaij}
(\alpha_i, \alpha_j) = d_i a_{ij}= d_j a_{ji}, \quad \forall i, j\in \hat{I}.
\end{equation}

Set $q_i :=q^{d_{i}}$ for any $i \in \hat{I}$. Now we can give the definition of the quantum affine superalgebra $\widehat{U}= \widehat{U}_q(\mathfrak{g})$ associated to $\mathfrak{g}$.
\begin{definition}
\label{def:superad}
$\widehat{U}$ is the unital associative superalgebra over $\F$ with generators
\[
    E_i, F_i, \quad \text{for} ~i \in \hat{I}, \qquad
    q^h,\quad \text{for} ~h \in P^\vee,
\]
and parities
\[
    |E_i| = |F_i| = |\alpha_i|,\quad \text{for} ~i \in \hat{I}, \qquad
    |q^h| = 0,\quad \text{for} ~h \in P^\vee,
\]
subject to the following relations: for $h,h' \in P^\vee$, $i,j \in \hat{I}$,
\begin{gather}
    q^0 = 1,\qquad
    q^h q^{h'} = q^{h+h'},
	\\ \label{KEF}
    q^h E_i = q^{\langle h, \alpha_i \rangle} E_i q^h,\qquad
    q^h F_i = q^{- \langle h, \alpha_i \rangle} F_i q^h,
    \\ \label{skunk}
    [E_i,F_j]
	= \delta_{ij} \frac{K_i - K_i^{-1}}{q_i-q_i^{-1}},\qquad
	\text{where } K_i := q^{d_ih_i},
    \\ \label{square}
    [E_i,E_j] = 0,\qquad
	[F_i,F_j] = 0,\qquad
	\text{if } a_{ij}=0,
    \\ \label{racoon}
    [E_i,[E_i,E_j]_{q_i}]_{q_i^{-1}}= 0 = [F_i,[F_i,F_j]_{q_i}]_{q_i^{-1}},\qquad
	\text{if } |a_{ij}|=1,\ |\alpha_i|=\overline{0},
    \\ \label{Serre}
        [[[E_{j-1},E_j]_{q_j},E_{j+1}]_{q_{j+1}}],E_j]=0,\quad
        [[[F_{j-1},F_j]_{q_j},F_{j+1}]_{q_{j+1}}],F_j]=0,
        \text{ if }  a_{jj}=0,
\end{gather}
where
\[
    [x,y]_a := xy-(-1)^{|x| |y|} a yx, \quad \text{for} ~a \in \F,
    \qquad \text{and} \qquad [x,y] := [x,y]_1,
\]
and all indices are taken modulo $N$.
\end{definition}

As in \cite[(6.4.3)]{Yam99}, the multiplication gives a vector superspace isomorphism, called a triangular decomposition of $\widehat{U}$,
\begin{align}\label{atriang}
\widehat{U}^{+}\otimes \widehat{U}^{0}\otimes \widehat{U}^{-}\cong \widehat{U},
\end{align}
where $\widehat{U}^{+}$ (resp. $\widehat{U}^{-}$) denotes the subalgebra of $\widehat{U}$ generated by $E_i$ (resp. $F_i$), $i \in \hat{I}$, and $\widehat{U}^{0}$ denotes the subalgebra of $\widehat{U}$ generated by $q^h$, $h \in P^\vee$. Let $\widehat{U}(\mathfrak{sl})$ be the subalgebra of $\widehat{U}$ generated by $E_i$, $F_i$, $K_i^{\pm 1}$ for $i \in \hat{I}$.

$\widehat{U}$ admits a Hopf superalgebra structure with the comultiplication as follows: for $i \in \hat{I}$, $h \in P^\vee$,
\begin{equation*}
    \Delta(E_i)=E_i\otimes 1+K_i\otimes E_i,\quad \Delta(F_i)=F_i\otimes K_i^{-1}+1 \otimes F_i, \quad \Delta(q^h)=q^h\otimes q^h.
\end{equation*}

Assume that there exists a parity-preserving permutation $\tau$ of $\hat{I}$ such that $\tau^{2}=\mathrm{id}$ and $\tau(i)=N-1-i$ for $0\leq i\leq r$. Given parameters $\varsigma_{i}\in \F$, we set $B_{i} :=F_{i}+\varsigma_{i}E_{\tau i}K_{i}^{-1}$, $\forall i\in \hat{I}$. We require that $\varsigma_{i}=\varsigma_{\tau i}$ for $1\leq i\leq r-1$ and $d_{i}=(-1)^{|\alpha_i|}d_{\tau i}$ for $0\leq i\leq r$. We set $\mathrm{h}_{0}=q^{2d_{N}\epsilon_{N}^\vee}$, $\mathrm{h}_{a}=q^{d_{a}\epsilon_{a}^\vee+d_{N-a}\epsilon_{N-a}^\vee}$ for $1\leq a\leq r$, $\mathrm{h}_{r+1}=q^{2d_{r+1}\epsilon_{r+1}^\vee}$.
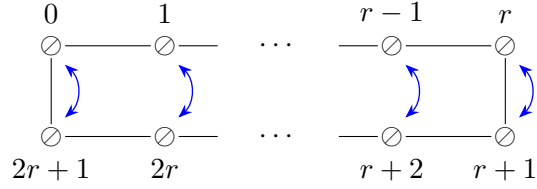
\begin{figure}[htbp]
\centering

\begin{tikzpicture}[
    scale=1,
    vertex/.style={
        inner sep=0pt,
        outer sep=1pt
    },
    >=Stealth
]

\node[vertex,label=above:$0$]
    (a0) at (0,0.6) {$\oslash$};

\node[vertex,label=above:$1$]
    (a1) at (1.5,0.6) {$\oslash$};

\node[vertex,label=above:$r-1$]
    (am) at (4.5,0.6) {$\oslash$};

\node[vertex,label=above:$r$]
    (ar) at (6,0.6) {$\oslash$};

\node[vertex,label=below:$2r+1$]
    (b0) at (0,-0.6) {$\oslash$};

\node[vertex,label=below:$2r$]
    (b1) at (1.5,-0.6) {$\oslash$};

\node[vertex,label=below:$r+2$]
    (bm) at (4.5,-0.6) {$\oslash$};

\node[vertex,label=below:$r+1$]
    (br) at (6,-0.6) {$\oslash$};

\draw (a0) -- (a1);
\draw (a1) -- (2.2,0.6);
\node at (3,0.6) {$\cdots$};
\draw (3.8,0.6) -- (am);
\draw (am) -- (ar);

\draw (b0) -- (b1);
\draw (b1) -- (2.2,-0.6);
\node at (3,-0.6) {$\cdots$};
\draw (3.8,-0.6) -- (bm);
\draw (bm) -- (br);

\draw (a0) -- (b0);
\draw (ar) -- (br);

\foreach \x in {0,1.5,4.5,6} {
    \draw[blue,<->,line width=0.5pt]
        (\x+0.18,0.35)
        .. controls (\x+0.40,0.15)
                     and (\x+0.40,-0.15)
        .. (\x+0.18,-0.35);
}

\end{tikzpicture}
\end{figure}

\noindent where
\begin{equation*}
\underset{i}{\oslash}\text{ or } \overset{i}{\oslash}
=
\begin{cases}
\underset{i}{\scalebox{1.8}{$\circ$}}\text{ or } \overset{i}{\scalebox{1.8}{$\circ$}},
  & \text{if } |\alpha_i|=\overline{0},\\[5pt]
\underset{i}{\otimes}\text{ or } \overset{i}{\otimes},
  & \text{if } |\alpha_i|=\overline{1}.
\end{cases}
\end{equation*}

\begin{definition}
\label{def:isuperad}
Let $\widehat{U}^{\jmath}$ be the subalgebra of $\widehat{U}$ generated by the elements $B_{i}$ for $i\in \hat{I}$, $\mathrm{h}_{j}^{\pm 1}$ for $0\leq j\leq r+1$.
\end{definition}

For $0\leq i\leq r$, we define $\mathrm{k}_{i} := \mathrm{h}_{i}\mathrm{h}_{i+1}^{-1}$ and $\mathrm{k}_{\tau i}= \mathrm{k}_{i}^{-1}$. We see that $\mathrm{k}_{i}=K_iK_{\tau i}^{-1}$ for any $i\in \hat{I}$. Then we have the following theorem.

\begin{thm}
\label{def:ipreseab}
$\widehat{U}^{\jmath}$ is a right coideal subalgebra of $\widehat{U}$. Moreover, it admits the following presentation: for $0\leq a, b\leq r+1$ and $i, j, k\in \hat{I}$,
\begin{align}\label{KEF1}
    \mathrm{h}_{a}\mathrm{h}_{b}=\mathrm{h}_{b}\mathrm{h}_{a},\qquad
    \mathrm{h}_{a}\mathrm{h}_{a}^{-1}=1=\mathrm{h}_{a}^{-1}\mathrm{h}_{a},
\end{align}
\begin{align}\label{skunk2}
    \mathrm{h}_{a}B_{i}=q^{- \langle d_{a}\epsilon_{a}^\vee+d_{N-a}\epsilon_{N-a}^\vee, \alpha_i \rangle}B_{i}\mathrm{h}_{a},
\end{align}
\begin{align}\label{square33}
    \text{for }1\leq i\leq r-1, \qquad [B_{\tau i},B_i]=
\frac{\varsigma_{\tau i}\mathrm{k}_{i} - \varsigma_{i}\mathrm{k}_{i}^{-1}}{q_i-q_i^{-1}},
\end{align}
\begin{align}\label{square3333}
    \text{for $i\not\equiv j\pm 1\pmod N$ and $j\neq \tau i$}, \qquad [B_{j},B_i]=0,
\end{align}
\begin{align*}
    \text{for }|\alpha_i|=\overline{0}, i\neq j \text{ and }i\sim j, \text{ we ~have }\\
    \end{align*}
\begin{align}\label{squareSerre5}
    [B_i,[B_i,B_j]_{q_i}]_{q_i^{-1}}=\begin{cases}
-(q_i+q_i^{-1})(q_i\varsigma_{\tau i}\mathrm{k}_{i}B_i+q_i^{-2}\varsigma_{i}\mathrm{k}_{i}^{-1}B_i)&\hspace{0.35cm}\text{if }j=\tau i,\\[0.3em]
0&\hspace{0.35cm}\text{otherwise},
\end{cases}
\end{align}
\begin{align*}
    \text{for }|\alpha_j|=\overline{1}, i=j-1 \text{ and }k=j+1, \text{ we have } \\
\end{align*}
\begin{align}\label{Serre6}
        &[[[B_{i},B_j]_{q_j},B_{k}]_{q_{k}}],B_j]\\
        =&\notag \begin{cases}
0&\hspace{0.35cm}\text{if }\{i, j, k\}\cap \{\tau i, \tau j, \tau k\}=\varnothing,\\[0.3em]
(-1)^{|\alpha_i|}\Big(q_{\tau j}^{-1}\varsigma_{\tau j}[B_{j},B_{i}]_{q_j}\mathrm{k}_{j}+\varsigma_{j}[B_{j},B_{i}]_{q_{\tau j}}\mathrm{k}_{j}^{-1}\Big)&\hspace{0.35cm}\text{if }k=\tau j,\\[0.3em]
 (-1)^{|\alpha_k|}\Big(q_j^{-1}\varsigma_{\tau j}[B_j,B_k]_{q_{\tau j}}\mathrm{k}_j+\varsigma_j[B_j,B_k]_{q_j}\mathrm{k}_{j}^{-1}\Big)&\hspace{0.35cm}\text{if }i=\tau j.
\end{cases}
\end{align}
\end{thm}
\begin{proof}
We have $\Delta(B_i)=B_i\otimes K_i^{-1}+ 1\otimes F_i+ \varsigma_{i}\mathrm{k}_{i}^{-1}\otimes E_{\tau i}K_{i}^{-1}\in \widehat{U}^{\jmath}\otimes \widehat{U}$ and $\Delta(\mathrm{h}_{j}^{\pm 1})=\mathrm{h}_{j}^{\pm 1}\otimes \mathrm{h}_{j}^{\pm 1}\in \widehat{U}^{\jmath}\otimes \widehat{U}$. Thus, we see that $\widehat{U}^{\jmath}$ is a right coideal subalgebra of $\widehat{U}$.

Recall that we assume $d_{i}=(-1)^{|\alpha_i|}d_{\tau i}$ for $0\leq i\leq r$ and $d_{0}=d_{N}$. By this we can easily obtain
\begin{align}\label{didni}
d_{i}=d_{N-i} \quad \text{for }0\leq i\leq r+1.
\end{align}
Using \eqref{didni}, we can verify that \eqref{skunk2} holds. By \eqref{car matrix} and \eqref{car matrix diaij}, we have
\begin{align}\label{alphaij}
(\alpha_i, \alpha_j)=(\alpha_{\tau i}, \alpha_{\tau j}),\quad \forall i, j\in \hat{I}.
\end{align}
Using \eqref{alphaij}, we can verify that \eqref{square3333} holds.

\eqref{square33}, \eqref{squareSerre5} and \eqref{Serre6} can be proved in a similar way as \cite[Prop. 3.9-3.11]{CS26}.

Assume that $\mathcal{A}$ is the abstract superalgebra generated by
\begin{align*}
     \tilde{b}_{i}\quad (i\in \hat{I}),\qquad \tilde{h}_{j}^{\pm 1}\quad (0\leq j\leq r+1),
\end{align*}
subject to the relations \eqref{KEF1}--\eqref{Serre6} with $B_{i}$ and $\mathrm{h}_{j}^{\pm 1}$ replaced by $\tilde{b}_{i}$ and $\tilde{h}_{j}^{\pm 1}$ and with parities $|\tilde{b}_{i}| = |\alpha_i|$ and $|\tilde{h}_{j}^{\pm 1}| =0$. The above calculations give a surjective homomorphism
\begin{align*}
     \Phi :\mathcal{A}\longrightarrow \widehat{U}^{\jmath}, \quad \tilde{b}_{i}\mapsto B_{i},\quad  \tilde{h}_{j}^{\pm 1}\mapsto \mathrm{h}_{j}^{\pm 1}.
\end{align*}
We need to show that $\Phi$ is injective.

Define a filtration on $\mathcal{A}$ by
\begin{align*}
     \mathrm{deg}(\tilde{b}_{i})=1, ~~\text{for }i\in \hat{I},\qquad \mathrm{deg}(\tilde{h}_{j}^{\pm 1})=0, ~~\text{for }0\leq j\leq r+1.
\end{align*}
It is easy to see that the leading homogeneous parts of \eqref{square33}--\eqref{Serre6} are precisely the defining Serre relations of $\widehat{U}^{-}$. Let $T$ be the free superalgebra generated by symbols $f_{i}$ with parities $|f_{i}| = |\alpha_i|$, $\forall i\in \hat{I}$, and let $V$ be the ideal generated by the defining Serre relations with replacing $F_{i}$ by $f_{i}$. By \cite[Prop. 6.8.1]{Yam99}, we have
\begin{align*}
    \widehat{U}^{-}\cong T/V.
\end{align*}

For any $J=(j_1,\ldots, j_{s})\in \hat{I}^{s}$, we define the length of $J$ by $|J|=s$, and introduce the following shorthand notations:
\begin{align*}
     F_{J}=F_{j_1}\cdots F_{j_s},\qquad f_{J}=f_{j_1}\cdots f_{j_s},\qquad \tilde{b}_{J}=\tilde{b}_{j_1}\cdots \tilde{b}_{j_s},\qquad B_{J}=B_{j_1}\cdots B_{j_s}.
\end{align*}
For $\lambda=(\lambda_{0},\ldots,\lambda_{r+1})\in \mathbb{Z}^{r+2}$, we set
\begin{align*}
\mathrm{h}^{\lambda} :=\mathrm{h}_{0}^{\lambda_{0}}\cdots \mathrm{h}_{r+1}^{\lambda_{r+1}}\qquad  \text{and} \qquad \tilde{h}^{\lambda} :=\tilde{h}_{0}^{\lambda_{0}}\cdots \tilde{h}_{r+1}^{\lambda_{r+1}}.
\end{align*}
For $t\in \mathbb{Z}_{\geq 0}$, let $\mathcal{F}^{t}(\mathcal{A})$ denote the subspace of $\mathcal{A}$ spanned by all elements $\tilde{h}^{\lambda}\tilde{b}_{K}$ for $\lambda\in \mathbb{Z}^{r+2}$ and $|K|\leq t$.

Choose $\mathcal{I}$ a fixed subset of $\cup_{s\in \mathbb{Z}_{\geq 0}} \hat{I}^{s}$ such that $\{F_{J}\:|\:J\in \mathcal{I}\}$ is a homogeneous monomial basis of $\widehat{U}^{-}$. For any $L\in \hat{I}^{t}$, homogeneity of the defining Serre relations makes sure that
\begin{align*}
     F_{L}=\sum\limits_{J\in \mathcal{I}\atop |J|=t}c_{LJ}F_{J}, \quad \text{ where $c_{LJ}\in \F$.}
\end{align*}
Thus, in $T$, we have
\begin{align*}
     f_{L}-\sum\limits_{J\in \mathcal{I}\atop |J|=t}c_{LJ}f_{J}=\sum\limits_{c}u_{c}S_{c}v_{c},
\end{align*}
where each $S_{c}$ is one of the homogeneous Serre relations and $u_{c}, v_{c}\in T$. Applying the algebra homomorphism $T\rightarrow \mathcal{A}$, ~~$f_{i}\mapsto \tilde{b}_{i}$, we obtain
\begin{align*}
     \tilde{b}_{L}-\sum\limits_{J\in \mathcal{I}\atop |J|=t}c_{LJ}\tilde{b}_{J}=\sum\limits_{c}u_{c}(\tilde{b})S_{c}(\tilde{b})v_{c}(\tilde{b}).
\end{align*}
By \eqref{square33}--\eqref{Serre6}, we have
\begin{align}\label{bjbj1}
\tilde{b}_{L}-\sum\limits_{J\in \mathcal{I}\atop |J|=t}c_{LJ}\tilde{b}_{J}\in \mathcal{F}^{t-1}(\mathcal{A}).
\end{align}
By induction on $t$, together with \eqref{KEF1} and \eqref{skunk2}, we can therefore show that $\mathcal{A}$ is spanned by the elements
\begin{align*}
     \tilde{h}^{\lambda}\tilde{b}_{J},\quad \lambda\in \mathbb{Z}^{r+2}, J\in \mathcal{I}.
\end{align*}
It remains to prove that the images of the above elements under $\Phi$ are linearly independent.

We now put a filtration on $\widehat{U}$:
\begin{align*}
     \mathrm{deg}_{F}(F_{i})=1, ~~~~\mathrm{deg}_{F}(E_{i})=0,~~~~\forall i\in \hat{I},\qquad \mathrm{deg}_{F}(q^h)=0, ~~\text{for} ~h \in P^\vee.
\end{align*}
Then $B_{i} :=F_{i}+\varsigma_{i}E_{\tau i}K_{i}^{-1}$ has leading term $F_{i}$ with $F$-degree $1$. Therefore, for any $J=(j_1,\ldots, j_{t})\in \hat{I}^{t}$, we have
\begin{align*}
     B_{J}=F_{J}+\text{terms of $F$-degree $< t$}.
\end{align*}
Indeed, the unique term of degree $t$ is obtained by choosing $F_{j_s}$ from every $B_{j_s}$. If at least one $E_{\tau j_s}$ is chosen, at most $t-1$ negative generators remain, and commuting an element $E_{s}$ with another $F_{s}$ can only replace an $(E_sF_s, F_sE_s)$ pair by a Cartan element, thereby lowering the $F$-degree further.

Now suppose that
\begin{align*}
     \sum\limits_{\lambda, J}c_{\lambda, J}\mathrm{h}^{\lambda}B_{J}=0.
\end{align*}
Choose $t$ maximal such that some $c_{\lambda, J}\neq 0$ and $|J|=t$. Taking the component of degree $t$ in the triangular decomposition yields
\begin{align*}
     \sum\limits_{\lambda, J\atop |J|=t}c_{\lambda, J}\mathrm{h}^{\lambda}F_{J}=0.
\end{align*}
It is easy to see that the $\mathrm{h}^{\lambda}$'s are distinct elements of $\widehat{U}^{0}$, and therefore they are linearly independent. Since the $F_{J}$'s are also linearly independent, by the triangular decomposition \eqref{atriang} we see that the products $\mathrm{h}^{\lambda}F_{J}$'s are linearly independent. Thus, we obtain $c_{\lambda, J}=0$ for all $\lambda$ and $|J|=t$. By induction on $F$-degree we show that $c_{\lambda, J}=0$ for all $\lambda$ and $J$. Thus, the elements
\begin{align*}
     \mathrm{h}^{\lambda}B_{J},\quad \lambda\in \mathbb{Z}^{r+2}, J\in \mathcal{I},
\end{align*}
are linearly independent in $\widehat{U}^{\jmath}$. Hence, $\mathrm{Ker}(\Phi)=\{0\}$, and so $\mathcal{A}\cong \widehat{U}^{\jmath}$. We are done.
\end{proof}


Let $\widehat{U}^{\jmath}(\mathfrak{sl})$ be the subalgebra of $\widehat{U}(\mathfrak{sl})$ generated by $B_{i}$ for $i\in \hat{I}$ and $\mathrm{k}_{j}^{\pm 1}$ for $0\leq j\leq r$. We have an analog of Theorem \ref{def:ipreseab} for $\widehat{U}^{\jmath}(\mathfrak{sl})$.
\begin{thm}
\label{def:ipreseabslma}
$\widehat{U}^{\jmath}(\mathfrak{sl})$ is a right coideal subalgebra of $\widehat{U}(\mathfrak{sl})$. Moreover, it admits the following presentation: for $0\leq a, b\leq r$ with $\mathrm{k}_{\tau a}= \mathrm{k}_{a}^{-1}$ and $i\in \hat{I}$,

\begin{align}\label{KEF1add}
    \mathrm{k}_{a}\mathrm{k}_{b}=\mathrm{k}_{b}\mathrm{k}_{a},\qquad
    \mathrm{k}_{a}\mathrm{k}_{a}^{-1}=1=\mathrm{k}_{a}^{-1}\mathrm{k}_{a},
\end{align}
	
\begin{align}\label{skunk2add}
    \mathrm{k}_{a}B_{i}=q^{-(\alpha_{a}, \alpha_i)+(\alpha_{\tau a}, \alpha_i)}B_{i}\mathrm{k}_{a},
\end{align}
together with the relations \eqref{square33}--\eqref{Serre6}.
\end{thm}
We shall call $\widehat{U}^{\jmath}$ and $\widehat{U}^{\jmath}(\mathfrak{sl})$ the $\jmath$-quantum affine superalgebras.

\section{The commutativity theorem}\label{comm}
In this section we construct an action on some infinite-dimensional tensor superspace of the affine Hecke algebra of type $C$ with three parameters, and prove the commutativity theorem as the first main step toward the desired Schur--Weyl duality.

For any $a, b\in \mathbb{Z}$ with $a\leq b$, we denote
$$[a..b] :=[a,b]\cap \mathbb{Z}=\{a, a+1,\ldots, b\}.$$

We fix $m, n\in \mathbb{Z}_{\geq 1}$ with $m\neq n$. Denote by $\mathbb{I}_{\overline{0}} :=[-m..m]$ and $\mathbb{I}_{\overline{1}} :=[-m-n..-m-1]\sqcup [m+1..m+n+1]$. Set $\mathbb{I} :=\mathbb{I}_{\overline{0}}\cup \mathbb{I}_{\overline{1}}$. The parity of an element $i\in \mathbb{I}$ is defined by
\begin{align*}
p(i)=\left\{
\begin{array}{ll}
\overline{0}&\hspace{0.35cm}\text{if $i\in \mathbb{I}_{\overline{0}}$},\\[0.3em]
\overline{1}&\hspace{0.35cm}\text{if $i\in \mathbb{I}_{\overline{1}}$}.
\end{array}
\right.
\end{align*}
Since each $a\in \mathbb{Z}$ has a unique expression $a=\overline{a}+b(2m+2n+2)$ for some $b\in \mathbb{Z}$ and $\overline{a}\in \mathbb{I}$, we set $p(a) :=p(\overline{a})$.

Let $V$ be the vector superspace over $\C$ with a basis $\{v_{1}, v_{2},\ldots,v_{2m+2n+2}\}$ and parities $|v_{j}|=p(j)$. In $\S$\ref{def:quanaffis}, associated to $V$ we have defined the quantum affine superalgebra and the associated $\jmath$-quantum affine superalgebra, which are also denoted by $\widehat{U}$ and $\widehat{U}^{\jmath}$, respectively. Note that in this setting, we have $m'=2m+1$, $n'=2n+1$, $r=m+n$ and $m\neq n$. Let $\mathbb{V}$ be an infinite dimensional superspace with an $\F$-basis $\{v_{j}\:|\:j\in \mathbb{Z}\}$ and parities $|v_{j}|=p(j)$. Then we have the following lemma.
\begin{lem}\label{uvactoni}
The following defines an action of $\widehat{U}$ on $\mathbb{V}$:
\begin{align*}
E_{i} v_{j+1}=\left\{
\begin{array}{ll}
v_j&\hspace{0.35cm}\text{if $j\equiv i\pmod N$},\\[0.3em]
0&\hspace{0.35cm}\text{otherwise}.
\end{array}
\right.
\end{align*}
\begin{align*}
F_{i} v_{j}=\left\{
\begin{array}{ll}
v_{j+1}&\hspace{0.35cm}\text{if $j\equiv i\pmod N$},\\[0.3em]
0&\hspace{0.35cm}\text{otherwise}.
\end{array}
\right.
\end{align*}
\begin{align*}
q^{\epsilon_{i}^{\vee}} v_{j}=\left\{
\begin{array}{ll}
qv_{j}&\hspace{0.35cm}\text{if $j\equiv i\pmod N$},\\[0.3em]
v_j&\hspace{0.35cm}\text{otherwise}.
\end{array}
\right.
\end{align*}
\end{lem}
\begin{proof}
Since $K_i= q^{d_ih_i}=q^{d_i\epsilon_{i}^{\vee}-d_{i+1}\epsilon_{i+1}^{\vee}}$, we have
\begin{align*}
K_{i} v_{j}=\left\{
\begin{array}{ll}
q^{d_{i}}v_j&\hspace{0.35cm}\text{if $j\equiv i$ (mod $N$)},\\[0.3em]
q^{-d_{i+1}}v_j&\hspace{0.35cm}\text{if $j\equiv i+1$ (mod $N$)},\\[0.3em]
v_j&\hspace{0.35cm}\text{otherwise}.
\end{array}
\right.
\end{align*}
Using this, we can easily verify that \eqref{skunk} holds. The other relations can be verified similarly, noting that the action of a product $E_{i_1}\cdots E_{i_s}$ or $F_{i_1}\cdots F_{i_s}$ can be nonzero only when the indices form a consecutive chain.
\end{proof}

From Lemma \ref{uvactoni}, we can obtain an explicit description of the action of $\widehat{U}^{\jmath}$ on $\mathbb{V}$.
\begin{lem}\label{uimahtvactoni}
The following defines an action of $\widehat{U}^{\jmath}$ on $\mathbb{V}$:
\begin{align*}
\mathrm{h}_{a} v_{j}=\left\{
\begin{array}{ll}
q^{2d_{a}}v_j&\hspace{0.35cm}\text{if $a=0, m+n+1$; $j\equiv a\pmod N$},\\[0.3em]
q^{d_{a}}v_j&\hspace{0.35cm}\text{if $1\leq a\leq m+n$; $j\equiv a$ or $-a\pmod N$},\\[0.3em]
v_j&\hspace{0.35cm}\text{otherwise},
\end{array}
\right.
\end{align*}
\begin{align*}
\text{for $i\neq 0, m+n+1$},\qquad B_{i} v_{j}=\left\{
\begin{array}{ll}
v_{j+1}&\hspace{0.35cm}\text{if $j\equiv i\pmod N$},\\[0.3em]
\varsigma_{i}v_{j-1}&\hspace{0.35cm}\text{if $j\equiv -i\pmod N$},\\[0.3em]
0&\hspace{0.35cm}\text{otherwise},
\end{array}
\right.
\end{align*}
\begin{align*}
\text{for $i= 0, m+n+1$},\qquad B_{i} v_{j}=\left\{
\begin{array}{ll}
v_{j+1}+\varsigma_{i}q^{-d_{i}}v_{j-1}&\hspace{0.35cm}\text{if $j\equiv i\pmod N$},\\[0.3em]
0&\hspace{0.35cm}\text{otherwise}.
\end{array}
\right.
\end{align*}
\end{lem}



Fix $d\in \mathbb{Z}_{\geq 2}$. Let $\mathbb{H}$ be the affine Hecke algebra of type $C_{d}$ with three parameters; that is, $\mathbb{H}$ is an $\F$-algebra generated by
\begin{align*}
T_{i}\quad \text{$(0\leq i\leq d-1)$,}\quad X_{a}^{\pm 1}\quad \text{$(1\leq a\leq d)$}
\end{align*}
with the following relations: for $1\leq a, b\leq d$ and for $0\leq i, j, k\leq d-1$,

$(1)$
\begin{align*}
X_{a}X_{a}^{-1}=1=X_{a}^{-1}X_{a},\qquad X_{a}X_{b}=X_{b}X_{a}.
\end{align*}

$(2)$
\begin{align*}
&(T_{0}-Q_{0}^{-1})(T_{0}+Q_{1})=0,\qquad (T_{i}-q^{-1})(T_{i}+q)=0\quad \text{$(i\neq 0)$,}\\
&T_{k}T_{k-1}T_{k}=T_{k-1}T_{k}T_{k-1},\quad \text{$(2\leq k\leq d-1)$,}\\
&T_{0}T_{1}T_{0}T_{1}=T_{1}T_{0}T_{1}T_{0},\qquad T_{i}T_{j}=T_{j}T_{i}\quad \text{$(|i-j|>1)$.}
\end{align*}

$(3)$
\begin{align*}
&T_{0}X_{1}^{-1}T_{0}=Q_{0}^{-1}Q_{1}X_{1}+(Q_{0}^{-1}Q_{1}-1)T_{0},\\
&T_{i}X_{i}T_{i}=X_{i+1}\quad \text{$(i\neq 0)$,}\qquad T_{i}X_{j}=X_{j}T_{i}\quad \text{$(j\neq i, i+1)$.}
\end{align*}

The algebra $\mathbb{H}$ contains a subalgebra $\mathbb{H}_{A}$ generated by $T_{1},\ldots,T_{d-1}$, $X_{1}^{\pm 1},\ldots,X_{d}^{\pm 1}$, which is an affine Hecke algebra of type $A$.

The tensor space $\mathbb{V}^{\otimes d}$ has an $\F$-basis $\{M_{f}\:|\:f\in \mathbb{Z}^{d}\}$, where
\begin{align*}
M_{f} :=v_{f_1}\otimes \cdots \otimes v_{f_{d}}\in \mathbb{V}^{\otimes d}\quad \text{for $f=(f_1,\ldots, f_d)\in \mathbb{Z}^{d}$.}
\end{align*}
We introduce linear operators $z_1,\ldots, z_{d}$ which act on $\mathbb{V}^{\otimes d}$ from the right as follows:
\begin{align*}
M_{f}z_{i} :=M_{(f_1,\ldots,f_{i-1},f_{i}+N,f_{i+1},\ldots,f_{d})}.
\end{align*}
For each $\mathbf{c}=(c_1,\ldots, c_d)\in \mathbb{Z}^{d}$, we define $Z_{\mathbf{c}}=z_{1}^{c_1}\cdots z_{d}^{c_d}$. Since each $f_i\in \mathbb{Z}$ has a unique expression $f_{i}=\overline{f}_{i}+c_{i}N$ for some $c_{i}\in \mathbb{Z}$ and $\overline{f}_{i}\in \mathbb{I}$, each element $M_{f}$ has a unique expression
\begin{align*}
M_{f}=M_{\overline{f}}Z_{\mathbf{c}}.
\end{align*}

For $1\leq i\leq d-1$ and $1\leq a\leq d$, we define
\begin{align*}
M_{\overline{f}}\cdot T_{i}=\left\{
\begin{array}{ll}
\frac{(-1)^{p(\overline{f}_{i})}(q^{-1}+q)+q^{-1}-q}{2}M_{\overline{f}}&\hspace{0.35cm}\text{if $\overline{f}_{i}=\overline{f}_{i+1}$},\\[0.3em]
(-1)^{p(\overline{f}_{i})p(\overline{f}_{i+1})}M_{\overline{f}\cdot s_{i}}&\hspace{0.35cm}\text{if $\overline{f}_{i}<\overline{f}_{i+1}$},\\[0.3em]
(-1)^{p(\overline{f}_{i})p(\overline{f}_{i+1})}M_{\overline{f}\cdot s_{i}}+(q^{-1}-q)M_{\overline{f}}&\hspace{0.35cm}\text{if $\overline{f}_{i}>\overline{f}_{i+1}$}.
\end{array}
\right.
\end{align*}

\begin{align*}
M_{f}\cdot X_{a}=M_{f}z_{a}^{-1}.
\end{align*}

For each $\mathbf{c}=(c_1,\ldots, c_d)\in \mathbb{Z}^{d}$, we define $X_{\mathbf{c}} :=X_{1}^{c_1}\cdots X_{d}^{c_d}$. An inductive argument gives the following formula
\begin{align}\label{xcti}
X_{\mathbf{c}}^{-1}T_{i}=T_{i}X_{\mathbf{c}\cdot s_{i}}^{-1}+(q-q^{-1})\frac{X_{i+1}(X_{\mathbf{c}}^{-1}-X_{\mathbf{c}\cdot s_{i}}^{-1})}{X_{i}-X_{i+1}}.
\end{align}
Applying \eqref{xcti}, we have
\begin{align*}
M_{f}\cdot T_{i}=&(M_{\overline{f}}\cdot X_{\mathbf{c}}^{-1})\cdot T_{i}=M_{\overline{f}}\cdot (X_{\mathbf{c}}^{-1}T_{i})\\
=&(M_{\overline{f}}\cdot T_{i})\cdot X_{\mathbf{c}\cdot s_{i}}^{-1}+(q-q^{-1})M_{\overline{f}}\cdot\frac{X_{i+1}(X_{\mathbf{c}}^{-1}-X_{\mathbf{c}\cdot s_{i}}^{-1})}{X_{i}-X_{i+1}},
\end{align*}
and therefore
\begin{align}\label{actions}
M_{f}\cdot T_{i}=\left\{
\begin{array}{ll}
\frac{(-1)^{p(\overline{f}_{i})}(q^{-1}+q)+q^{-1}-q}{2}M_{\overline{f}}Z_{\mathbf{c}\cdot s_{i}}+(q-q^{-1})M_{\overline{f}}\frac{z_{i}(Z_{\mathbf{c}}-Z_{\mathbf{c}\cdot s_{i}})}{z_{i+1}-z_{i}}
&\hspace{0.35cm}\text{if $\overline{f}_{i}=\overline{f}_{i+1}$},\\[0.3em]
(-1)^{p(\overline{f}_{i})p(\overline{f}_{i+1})}M_{\overline{f}\cdot s_{i}}Z_{\mathbf{c}\cdot s_{i}}+(q-q^{-1})M_{\overline{f}}\frac{z_{i}(Z_{\mathbf{c}}-Z_{\mathbf{c}\cdot s_{i}})}{z_{i+1}-z_{i}}&\hspace{0.35cm}\text{if $\overline{f}_{i}<\overline{f}_{i+1}$},\\[0.3em]
(-1)^{p(\overline{f}_{i})p(\overline{f}_{i+1})}M_{\overline{f}\cdot s_{i}}Z_{\mathbf{c}\cdot s_{i}}+(q-q^{-1})M_{\overline{f}}\frac{z_{i}Z_{\mathbf{c}}- z_{i+1}Z_{\mathbf{c}\cdot s_{i}}}{z_{i+1}-z_{i}}
&\hspace{0.35cm}\text{if $\overline{f}_{i}>\overline{f}_{i+1}$}.
\end{array}
\right.
\end{align}


From relations (3), we have, for any $c\in \mathbb{Z}$,
\begin{align}\label{xact0}
X_{1}^{-c}T_{0}=T_{0}X_{1}^{c}+(Q_{1}-Q_{0}^{-1})\frac{X_{1}(X_{1}^{-c}-X_{1}^{c})}{X_{1}^{-1}-X_{1}}+(Q_{0}^{-1}Q_{1}-1)\frac{X_{1}^{-c}-X_{1}^{c}}{X_{1}^{-1}-X_{1}}.
\end{align}
Recall that each $f_1\in \mathbb{Z}$ has a unique expression $f_{1}=\overline{f}_{1}+jN$ for some $j\in \mathbb{Z}$ and $\overline{f}_{1}\in \mathbb{I}$. We define
\begin{align*}
&(v_{\overline{f}_{1}}\otimes v_{f_2}\otimes\cdots \otimes v_{f_{d}})\cdot T_{0}\\
=&\left\{
\begin{array}{ll}
(-1)^{p(\overline{f}_{1})}v_{-\overline{f}_{1}}\otimes v_{f_2}\otimes\cdots \otimes v_{f_{d}}&\hspace{0.35cm}\text{if $0<\overline{f}_{1}\leq m+n+1$},\\[0.3em]
Q_{0}^{-1}v_{0}\otimes v_{f_2}\otimes\cdots \otimes v_{f_{d}}&\hspace{0.35cm}\text{if $\overline{f}_{1}=0$},\\[0.3em]
(-1)^{p(\overline{f}_{1})}Q_{0}^{-1}Q_{1}v_{-\overline{f}_{1}}\otimes v_{f_2}\otimes\cdots \otimes v_{f_{d}}+(Q_{0}^{-1}-Q_{1})v_{\overline{f}_{1}}\otimes v_{f_2}\otimes\cdots \otimes v_{f_{d}}
&\hspace{0.35cm}\text{if $-m-n\leq \overline{f}_{1}<0$}.
\end{array}
\right.
\end{align*}
Applying \eqref{xact0}, we obtain
\begin{align*}
M_{f}\cdot T_{0}=&((v_{\overline{f}_{1}}\otimes v_{f_2}\otimes\cdots \otimes v_{f_{d}})\cdot X_{1}^{-j})\cdot T_{0}=(v_{\overline{f}_{1}}\otimes v_{f_2}\otimes\cdots \otimes v_{f_{d}})\cdot (X_{1}^{-j}T_{0})\\
=&((v_{\overline{f}_{1}}\otimes v_{f_2}\otimes\cdots \otimes v_{f_{d}})\cdot T_{0})\cdot X_{1}^{j}+(Q_{1}-Q_{0}^{-1})(v_{\overline{f}_{1}}\otimes v_{f_2}\otimes\cdots \otimes v_{f_{d}})\cdot\frac{X_{1}(X_{1}^{-j}-X_{1}^{j})}{X_{1}^{-1}-X_{1}}\\
&+(Q_{0}^{-1}Q_{1}-1)(v_{\overline{f}_{1}}\otimes v_{f_2}\otimes\cdots \otimes v_{f_{d}})\cdot\frac{X_{1}^{-j}-X_{1}^{j}}{X_{1}^{-1}-X_{1}},
\end{align*}
and hence
\begin{align*}
&M_{f}\cdot T_{0}=\\
&(-1)^{p(\overline{f}_{1})}(v_{-\overline{f}_{1}}\otimes v_{f_2}\otimes\cdots \otimes v_{f_{d}})z_{1}^{-j}+(Q_{1}-Q_{0}^{-1})(v_{\overline{f}_{1}}\otimes v_{f_2}\otimes\cdots \otimes v_{f_{d}})(\sum_{l=1}^{j}z_{1}^{j-2l})\\
&+(Q_{0}^{-1}Q_{1}-1)(v_{\overline{f}_{1}}\otimes v_{f_2}\otimes\cdots \otimes v_{f_{d}})(\sum_{l=1}^{j}z_{1}^{j+1-2l})\hspace{0.1cm}\quad\text{if $0<\overline{f}_{1}\leq m+n+1$ and $j\geq 0$},\\[0.3em]
\end{align*}
\begin{align*}
&M_{f}\cdot T_{0}=\\
&(-1)^{p(\overline{f}_{1})}(v_{-\overline{f}_{1}}\otimes v_{f_2}\otimes\cdots \otimes v_{f_{d}})z_{1}^{-j}+(Q_{0}^{-1}-Q_{1})(v_{\overline{f}_{1}}\otimes v_{f_2}\otimes\cdots \otimes v_{f_{d}})(\sum_{l=1}^{-j}z_{1}^{-j-2l})\\
&+(1-Q_{0}^{-1}Q_{1})(v_{\overline{f}_{1}}\otimes v_{f_2}\otimes\cdots \otimes v_{f_{d}})(\sum_{l=1}^{-j}z_{1}^{-j+1-2l})\hspace{0.1cm}\quad\text{if $0<\overline{f}_{1}\leq m+n+1$ and $j<0$},\\[0.3em]
\end{align*}
\begin{align*}
&M_{f}\cdot T_{0}=\\
&Q_{0}^{-1}(v_{0}\otimes v_{f_2}\otimes\cdots \otimes v_{f_{d}})z_{1}^{-j}+(Q_{1}-Q_{0}^{-1})(v_{0}\otimes v_{f_2}\otimes\cdots \otimes v_{f_{d}})(\sum_{l=1}^{j}z_{1}^{j-2l})\\
&+(Q_{0}^{-1}Q_{1}-1)(v_{0}\otimes v_{f_2}\otimes\cdots \otimes v_{f_{d}})(\sum_{l=1}^{j}z_{1}^{j+1-2l})\hspace{0.1cm}\quad\text{if $\overline{f}_{1}=0$ and $j\geq 0$},\\[0.3em]
\end{align*}
\begin{align*}
&M_{f}\cdot T_{0}=\\
&Q_{0}^{-1}(v_{0}\otimes v_{f_2}\otimes\cdots \otimes v_{f_{d}})z_{1}^{-j}+(Q_{0}^{-1}-Q_{1})(v_{0}\otimes v_{f_2}\otimes\cdots \otimes v_{f_{d}})(\sum_{l=1}^{-j}z_{1}^{-j-2l})\\
&+(1-Q_{0}^{-1}Q_{1})(v_{0}\otimes v_{f_2}\otimes\cdots \otimes v_{f_{d}})(\sum_{l=1}^{-j}z_{1}^{-j+1-2l})\hspace{0.1cm}\quad\text{if $\overline{f}_{1}=0$ and $j< 0$},\\[0.3em]
\end{align*}
\begin{align*}
&M_{f}\cdot T_{0}=\\
&(-1)^{p(\overline{f}_{1})}Q_{0}^{-1}Q_{1}(v_{-\overline{f}_{1}}\otimes v_{f_2}\otimes\cdots \otimes v_{f_{d}})z_{1}^{-j}+(Q_{1}-Q_{0}^{-1})(v_{\overline{f}_{1}}\otimes v_{f_2}\otimes\cdots \otimes v_{f_{d}})(\sum_{l=1}^{j-1}z_{1}^{j-2l})\\
&+(Q_{0}^{-1}Q_{1}-1)(v_{\overline{f}_{1}}\otimes v_{f_2}\otimes\cdots \otimes v_{f_{d}})(\sum_{l=1}^{j}z_{1}^{j+1-2l})\hspace{0.1cm}\quad\text{if $-m-n\leq \overline{f}_{1}<0$ and $j> 0$},\\[0.3em]
\end{align*}
\begin{align*}
&M_{f}\cdot T_{0}=\\
&(-1)^{p(\overline{f}_{1})}Q_{0}^{-1}Q_{1}(v_{-\overline{f}_{1}}\otimes v_{f_2}\otimes\cdots \otimes v_{f_{d}})z_{1}^{-j}+(Q_{0}^{-1}-Q_{1})(v_{\overline{f}_{1}}\otimes v_{f_2}\otimes\cdots \otimes v_{f_{d}})(\sum_{l=0}^{-j}z_{1}^{-j-2l})\\
&+(1-Q_{0}^{-1}Q_{1})(v_{\overline{f}_{1}}\otimes v_{f_2}\otimes\cdots \otimes v_{f_{d}})(\sum_{l=1}^{-j}z_{1}^{-j+1-2l})\hspace{0.1cm}\quad\text{if $-m-n\leq \overline{f}_{1}<0$ and $j\leq0$},\\[0.3em]
\end{align*}
\begin{rem}
Assume that
\begin{align*}
v_{m+n+1}\cdot T_{0}=A v_{-m-n-1},\qquad v_{-m-n-1}\cdot T_{0}=B v_{m+n+1}+(Q_{0}^{-1}-Q_{1})v_{-m-n-1}.
\end{align*}
In order to make $T_{0}^{2}=Q_{0}^{-1}Q_{1}+(Q_{0}^{-1}-Q_{1})T_{0}$ valid, we must have $AB=Q_{0}^{-1}Q_{1}$. On the other hand, in order to make sure that $X_{1}T_{0}=T_{0}X_{1}^{-1}+(Q_{0}^{-1}-Q_{1})X_{1}+(1-Q_{0}^{-1}Q_{1})$, we must have $A+(1-Q_{0}^{-1}Q_{1})=B$ noting that $v_{-m-n-1}=v_{m+n+1}\cdot X_{1}$. Thus, we obtain $A=-1$ and $B=-Q_{0}^{-1}Q_{1}$, or $A=Q_{0}^{-1}Q_{1}$ and $B=1$. In \cite{FLLLWW20}, they use the latter case, while we adopt the former one here.
\end{rem}

Making use of the above formulas we can obtain the following result.
\begin{thm}
\label{the:ipreseab}
The above formula \eqref{actions} and the $T_{0}$-action formulas define an action of $\mathbb{H}$ on $\mathbb{V}^{\otimes d}$.
\end{thm}
\begin{proof}
We first verify that the $T_{i}$'s for $1\leq i\leq d-1$ satisfy the quadratic relations $(T_{i}-q^{-1})(T_{i}+q)=0$. Since $T_i$ has no effect on the positions other than the $i$ and $(i+1)$-th positions, we assume that $d=2$. When $\overline{f}_{i}=\overline{f}_{i+1}$, we have
\begin{align*}
&[(v_{\overline{f}_{i}}\otimes v_{\overline{f}_{i+1}})Z_{\mathbf{c}}]\cdot T_{i}^{2}\\
=&\bigg[\frac{(-1)^{p(\overline{f}_{i})}(q^{-1}+q)+q^{-1}-q}{2}(v_{\overline{f}_{i}}\otimes v_{\overline{f}_{i+1}})Z_{\mathbf{c}\cdot s_{i}}+(q-q^{-1})(v_{\overline{f}_{i}}\otimes v_{\overline{f}_{i+1}})\frac{z_{i}(Z_{\mathbf{c}}-Z_{\mathbf{c}\cdot s_{i}})}{z_{i+1}-z_{i}}\bigg]\cdot T_{i}\\
=&\frac{(-1)^{p(\overline{f}_{i})}(q^{-1}+q)+q^{-1}-q}{2}\bigg[\frac{(-1)^{p(\overline{f}_{i})}(q^{-1}+q)+q^{-1}-q}{2}(v_{\overline{f}_{i}}\otimes v_{\overline{f}_{i+1}})Z_{\mathbf{c}}\\
&+(q-q^{-1})(v_{\overline{f}_{i}}\otimes v_{\overline{f}_{i+1}})\frac{z_{i}(Z_{\mathbf{c}\cdot s_{i}}-Z_{\mathbf{c}})}{z_{i+1}-z_{i}}\bigg]\\
&+(q-q^{-1})\bigg[\frac{(-1)^{p(\overline{f}_{i})}(q^{-1}+q)+q^{-1}-q}{2}(v_{\overline{f}_{i}}\otimes v_{\overline{f}_{i+1}})\frac{z_{i+1}(Z_{\mathbf{c}\cdot s_{i}}-Z_{\mathbf{c}})}{z_{i}-z_{i+1}}\\
&+(q-q^{-1})(v_{\overline{f}_{i}}\otimes v_{\overline{f}_{i+1}}) \frac{z_{i}\cdot \frac{z_{i}(Z_{\mathbf{c}}-Z_{\mathbf{c}\cdot s_{i}})}{z_{i+1}-z_{i}}- z_{i}\cdot \frac{z_{i+1}(Z_{\mathbf{c}\cdot s_{i}}-Z_{\mathbf{c}})}{z_{i}-z_{i+1}}}{z_{i+1}-z_{i}}\bigg]\\
=&(v_{\overline{f}_{i}}\otimes v_{\overline{f}_{i+1}})Z_{\mathbf{c}}+(q^{-1}-q)\cdot\bigg[\frac{(-1)^{p(\overline{f}_{i})}(q^{-1}+q)+q^{-1}-q}{2}(v_{\overline{f}_{i}}\otimes v_{\overline{f}_{i+1}})Z_{\mathbf{c}\cdot s_{i}}\\
&+(q-q^{-1})(v_{\overline{f}_{i}}\otimes v_{\overline{f}_{i+1}})\frac{z_{i}(Z_{\mathbf{c}}-Z_{\mathbf{c}\cdot s_{i}})}{z_{i+1}-z_{i}}\bigg]\\
=&[(v_{\overline{f}_{i}}\otimes v_{\overline{f}_{i+1}})Z_{\mathbf{c}}]\cdot (1+(q^{-1}-q)T_{i}).
\end{align*}
The cases when $\overline{f}_{i}< \overline{f}_{i+1}$ or $\overline{f}_{i}> \overline{f}_{i+1}$ can be proved similarly.

The braid relation $T_{i}T_{i+1}T_{i}=T_{i+1}T_{i}T_{i+1}$, for $1\leq i\leq d-2$, follows from a lengthy calculation.

The relations $(T_{0}-Q_{0}^{-1})(T_{0}+Q_{1})=0$ and $T_{0}T_{1}T_{0}T_{1}=T_{1}T_{0}T_{1}T_{0}$ also hold by a lengthy calculation.
\end{proof}

For $f=(f_1,\ldots, f_d)\in \mathbb{Z}^{d}$, define
\begin{align*}
\gamma_{a}(f) :=\#\{1\leq s\leq d\:|\:f_{s}\equiv a ~~~~(\text{mod }N)\}.
\end{align*}
For $i \in \hat{I}$, $t\in \mathbb{Z}$ and $1\leq a\leq d$, put
\begin{align*}
\kappa_{i}(t) :=q^{d_{i}\delta_{t\equiv i}-d_{i+1}\delta_{t\equiv i+1}} \quad \text{ and }\quad\epsilon_{i, a}(f) :=(-1)^{|\alpha_i|\sum_{s< a}p(f_{s})}.
\end{align*}

Using the iterated coproduct, we can define an action of $\widehat{U}$ on $\mathbb{V}^{\otimes d}$ from the left as follows.
\begin{lem}\label{auivdaddu}
The following defines an action of $\widehat{U}$ on $\mathbb{V}^{\otimes d}$: for $i \in \hat{I}$ and $M_{f}\in \mathbb{V}^{\otimes d}$,
\begin{align*}
E_{i}M_{f}=\sum\limits_{1\leq a\leq d\atop f_{a}\equiv i+1}\epsilon_{i, a}(f)\Big(\prod_{r< a}\kappa_{i}(f_{r})\Big)M_{f-e_{a}},
\end{align*}
\begin{align*}
F_{i}M_{f}=\sum\limits_{1\leq a\leq d\atop f_{a}\equiv i}\epsilon_{i, a}(f)\Big(\prod_{r> a}\kappa_{i}(f_{r})^{-1}\Big)M_{f+e_{a}},
\end{align*}
\begin{align*}
q^{\epsilon_{i}^{\vee}}M_{f}=q^{\sum\limits_{j=1}^{d} \delta_{f_{j}\equiv i}}M_{f}=q^{\gamma_{i}(f)}M_{f},
\end{align*}
where $M_{f\pm e_{a}}=v_{f_1}\otimes \cdots \otimes v_{f_{a}\pm 1}\otimes\cdots\otimes v_{f_{d}}$.
\end{lem}
\begin{proof}
They follow from
\begin{align*}
\Delta^{(d-1)}(E_{i})=&\sum\limits_{a=1}^{d}(K_{i})^{\otimes(a-1)}\otimes E_{i}\otimes 1^{\otimes(d-a)},\quad \Delta^{(d-1)}(F_{i})=&\sum\limits_{a=1}^{d}1^{\otimes(a-1)}\otimes F_{i}\otimes (K_{i}^{-1})^{\otimes(d-a)}.
\end{align*}
\end{proof}

Moreover, we have the following result.
\begin{prop}\label{hatu}
The left action of $\widehat{U}$ and the right one of $\mathbb{H}_{A}$ on $\mathbb{V}^{\otimes d}$ commute.
\end{prop}

By restriction we obtain an action of $\widehat{U}^{\jmath}$ on $\mathbb{V}^{\otimes d}$.
\begin{lem}\label{auivd}
The following defines an action of $\widehat{U}^{\jmath}$ on $\mathbb{V}^{\otimes d}$: for $0\leq a\leq m+n+1$, $i \in \hat{I}$ and $M_{f}\in \mathbb{V}^{\otimes d}$,
\begin{align}\label{hamf}
\mathrm{h}_{a}M_{f}=q^{d_{a}(\gamma_{a}(f)+\gamma_{-a}(f))}M_{f},
\end{align}
and
\begin{align}\label{bimf}
B_{i}M_{f}=&\sum\limits_{1\leq a\leq d\atop f_{a}\equiv i}\epsilon_{i, a}(f)\Big(\prod_{r> a}\kappa_{i}(f_{r})^{-1}\Big)M_{f+e_{a}}\\
&\notag+\varsigma_{i}\sum\limits_{1\leq a\leq d\atop f_{a}\equiv -i}\epsilon_{i, a}(f)\Big(\prod_{r< a}\kappa_{i}(f_{r})^{-1}\kappa_{\tau i}(f_{r})\Big)\kappa_{i}(f_{a})^{-1}\Big(\prod_{r> a}\kappa_{i}(f_{r})^{-1}\Big)M_{f-e_{a}}.
\end{align}
\end{lem}
\begin{proof}
Since $\Delta^{(d-1)}(\mathrm{h}_{a})=\mathrm{h}_{a}^{\otimes d}$ and $d_{a}=d_{N-a}$, we obtain \eqref{hamf} by Lemma \ref{uimahtvactoni}.

It is easy to verify that the iterated coproduct gives the following formula:
\begin{align*}
\Delta^{(d-1)}(B_{i})=&\sum\limits_{a=1}^{d}1^{\otimes(a-1)}\otimes F_{i}\otimes (K_{i}^{-1})^{\otimes(d-a)}\\
&+\varsigma_{i}\sum\limits_{a=1}^{d}(K_{\tau i}K_{i}^{-1})^{\otimes(a-1)}\otimes E_{\tau i}K_{i}^{-1}\otimes (K_{i}^{-1})^{\otimes(d-a)}.
\end{align*}
By Lemma \ref{uvactoni}, we can obtain \eqref{bimf}.
\end{proof}

By restriction, we further obtain an action of $\widehat{U}^{\jmath}(\mathfrak{sl})$ on $\mathbb{V}^{\otimes d}$ from Lemma \ref{auivd}. Then we have the following result.
\begin{thm}\label{hatcuth}
Assume that the parameters of $\widehat{U}^{\jmath}$ are specialized by
$$ \varsigma_0=qQ_1^{-1}, \qquad \varsigma_{m+n+1} =q^{-1}Q_0^{-1}Q_1, \qquad \varsigma_{N-1}=Q_0,  $$
and
$$ \varsigma_i=(-1)^{|\alpha_i|} \qquad \text{for }i\in \hat{I}\setminus \{0,m+n+1,N-1\}.$$
Then on $\mathbb{V}^{\otimes d}$, the left action of $\widehat{U}^{\jmath}$, and hence the action of $\widehat{U}^{\jmath}(\mathfrak{sl})$, commutes with the right one of $\mathbb{H}$.
\end{thm}
\begin{proof}
In view of Prop. \ref{hatu}, it suffices to show that the $T_{0}$-action commutes with the $\widehat{U}^{\jmath}$-action. Since
\begin{align*}
\Delta^{(d-1)}(B_{i})=&B_{i}\otimes K_{i}^{-1}\otimes \cdots \otimes K_{i}^{-1}+1\otimes (\cdots)+\varsigma_{i}K_{\tau i}K_{i}^{-1}\otimes \cdots,
\end{align*}
and $T_{0}$ has an effect only on the first position, it suffices to consider the case when $d=1$. That is, we need to show that for $i \in \hat{I}$, $k\in \mathbb{I}$, $j\in \mathbb{Z}$ and $0\leq a\leq m+n+1$,
\begin{align*}
(B_{i}\cdot v_{k}z^{j})\cdot T_{0}=B_{i}\cdot (v_{k}z^{j}\cdot T_{0}),\qquad (\mathrm{h}_{a}\cdot v_{k}z^{j})\cdot T_{0}=\mathrm{h}_{a}\cdot (v_{k}z^{j}\cdot T_{0}).
\end{align*}
From Lemma \ref{uimahtvactoni}, it is easy to see that the second equality holds. Next we verify the first one.

When $1\leq i\leq m+n$, $k=i$ and $j\geq 0$, we have
\begin{align*}
&(B_{i}\cdot v_{k}z^{j})\cdot T_{0}=v_{k+1}z^{j}\cdot T_{0}\\
=&(-1)^{p(k+1)}v_{-k-1}z^{-j}+(Q_{1}-Q_{0}^{-1})v_{k+1}\Big(\sum_{l=1}^{j}z^{j-2l}\Big)
+(Q_{0}^{-1}Q_{1}-1)v_{k+1}\Big(\sum_{l=1}^{j}z^{j+1-2l}\Big),
\end{align*}
\begin{align*}
&B_{i}\cdot (v_{k}z^{j}\cdot T_{0})\\
=&B_{i}\cdot \Big[(-1)^{p(k)}v_{-k}z^{-j}+(Q_{1}-Q_{0}^{-1})v_{k}\Big(\sum_{l=1}^{j}z^{j-2l}\Big)
+(Q_{0}^{-1}Q_{1}-1)v_{k}\Big(\sum_{l=1}^{j}z^{j+1-2l}\Big)\Big]\\
=&(-1)^{p(k)}\varsigma_i v_{-k-1}z^{-j}+(Q_{1}-Q_{0}^{-1})v_{k+1}\Big(\sum_{l=1}^{j}z^{j-2l}\Big)
+(Q_{0}^{-1}Q_{1}-1)v_{k+1}\Big(\sum_{l=1}^{j}z^{j+1-2l}\Big).
\end{align*}
Since $\varsigma_i=(-1)^{|\alpha_i|}$ for $1\leq i\leq m+n$, we have $(-1)^{p(k)}\varsigma_i=(-1)^{p(k+1)}$, and hence $(B_{i}\cdot v_{k}z^{j})\cdot T_{0}=B_{i}\cdot (v_{k}z^{j}\cdot T_{0})$. The case when $k=i$ and $j< 0$ can be verified similarly.

When $k=-i\in [-m-n+1..-1]$ and $j> 0$, we have
\begin{align*}
&(B_{i}\cdot v_{k}z^{j})\cdot T_{0}=(\varsigma_i v_{k-1}z^{j})\cdot T_{0}\\
=&\varsigma_i\Big[(-1)^{p(k-1)}Q_{0}^{-1}Q_{1}v_{-k+1}z^{-j}+(Q_{1}-Q_{0}^{-1})v_{k-1}\Big(\sum_{l=1}^{j-1}z^{j-2l}\Big)
+(Q_{0}^{-1}Q_{1}-1)v_{k-1}\Big(\sum_{l=1}^{j}z^{j+1-2l}\Big)\Big],
\end{align*}
\begin{align*}
&B_{i}\cdot (v_{k}z^{j}\cdot T_{0})\\
=&B_{i}\cdot \Big[(-1)^{p(k)}Q_{0}^{-1}Q_{1}v_{-k}z^{-j}+(Q_{1}-Q_{0}^{-1})v_{k}\Big(\sum_{l=1}^{j-1}z^{j-2l}\Big)
+(Q_{0}^{-1}Q_{1}-1)v_{k}\Big(\sum_{l=1}^{j}z^{j+1-2l}\Big)\Big]\\
=&(-1)^{p(k)}Q_{0}^{-1}Q_{1}v_{-k+1}z^{-j}+\varsigma_i(Q_{1}-Q_{0}^{-1})v_{k-1}\Big(\sum_{l=1}^{j-1}z^{j-2l}\Big)+\varsigma_i(Q_{0}^{-1}Q_{1}-1)v_{k-1}\Big(\sum_{l=1}^{j}z^{j+1-2l}\Big).
\end{align*}
Since $\varsigma_i=(-1)^{p(i)+p(i+1)}=(-1)^{p(-i)+p(-i-1)}=(-1)^{p(k)+p(k-1)}$, we have the desired equality. The case when $k=-i\in [-m-n+1..-1]$ and $j\leq 0$ can be proved similarly.

When $1\leq i\leq m+n$, $k\neq i$ or $-i$, we have $(B_{i}\cdot v_{k}z^{j})\cdot T_{0}=B_{i}\cdot (v_{k}z^{j}\cdot T_{0})=0$.

When $i=m+n$ and $k=-i$, and $j> 0$, we have
\begin{align*}
&(B_{i}\cdot v_{k}z^{j})\cdot T_{0}=(\varsigma_i v_{m+n+1}z^{j-1})\cdot T_{0}\\
=&\varsigma_i\Big[(-1)^{p(m+n+1)}v_{-m-n-1}z^{-j+1}+(Q_{1}-Q_{0}^{-1})v_{m+n+1}\Big(\sum_{l=1}^{j-1}z^{j-1-2l}\Big)
+(Q_{0}^{-1}Q_{1}-1)v_{m+n+1}\Big(\sum_{l=1}^{j-1}z^{j-2l}\Big)\Big],
\end{align*}
\begin{align*}
&B_{i}\cdot (v_{k}z^{j}\cdot T_{0})\\
=&(-1)^{p(-m-n)}Q_{0}^{-1}Q_{1}v_{m+n+1}z^{-j}+\varsigma_i(Q_{1}-Q_{0}^{-1})v_{-m-n-1}\Big(\sum_{l=1}^{j-1}z^{j-2l}\Big)\\
&+\varsigma_i(Q_{0}^{-1}Q_{1}-1)v_{-m-n-1}\Big(\sum_{l=1}^{j}z^{j+1-2l}\Big).
\end{align*}
Since $(-1)^{p(-m-n)}=(-1)^{p(m+n+1)}=-1$ and $\varsigma_{m+n}=1$, we have the desired equality. The case when $k=-i=-m-n$ and $j\leq 0$ can be proved similarly.

The case when $m+n+2\leq i\leq 2m+2n$ can be proved in a similar way as the case $1\leq i\leq m+n$.

When $k=i=0$ and $j> 0$, we have
\begin{align*}
&(B_{0}\cdot v_{0}z^{j})\cdot T_{0}=(v_{1}+\varsigma_{0}q^{-d_{0}}v_{-1})z^{j}\cdot T_{0}\\
=&(-1)^{p(1)}v_{-1}z^{-j}+(Q_{1}-Q_{0}^{-1})v_{1}\Big(\sum_{l=1}^{j}z^{j-2l}\Big)
+(Q_{0}^{-1}Q_{1}-1)v_{1}\Big(\sum_{l=1}^{j}z^{j+1-2l}\Big)\\
&+\varsigma_{0}q^{-d_{0}}\Big[(-1)^{p(-1)}Q_{0}^{-1}Q_{1}v_{1}z^{-j}+(Q_{1}-Q_{0}^{-1})v_{-1}\Big(\sum_{l=1}^{j-1}z^{j-2l}\Big)
+(Q_{0}^{-1}Q_{1}-1)v_{-1}\Big(\sum_{l=1}^{j}z^{j+1-2l}\Big)\Big],
\end{align*}
\begin{align*}
&B_{0}\cdot (v_{0}z^{j}\cdot T_{0})=B_{0}\cdot \Big[Q_{0}^{-1}v_{0}z^{-j}+(Q_{1}-Q_{0}^{-1})v_{0}\Big(\sum_{l=1}^{j}z^{j-2l}\Big)+(Q_{0}^{-1}Q_{1}-1)v_{0}\Big(\sum_{l=1}^{j}z^{j+1-2l}\Big)\Big]\\
=&Q_{0}^{-1}(v_{1}+\varsigma_{0}q^{-d_{0}}v_{-1})z^{-j}+(Q_{1}-Q_{0}^{-1})(v_{1}+\varsigma_{0}q^{-d_{0}}v_{-1})\Big(\sum_{l=1}^{j}z^{j-2l}\Big)\\
&+(Q_{0}^{-1}Q_{1}-1)(v_{1}+\varsigma_{0}q^{-d_{0}}v_{-1})\Big(\sum_{l=1}^{j}z^{j+1-2l}\Big).
\end{align*}
Since $q^{-1}\varsigma_0 Q_1=1$, we easily obtain $(B_{0}\cdot v_{0}z^{j})\cdot T_{0}=B_{0}\cdot (v_{0}z^{j}\cdot T_{0})$. The other cases can be proved similarly.

When $k=i=m+n+1$ and $j\geq 0$, we have
\begin{align*}
&(B_{i}\cdot v_{i}z^{j})\cdot T_{0}=(v_{-m-n}z+\varsigma_{i}q^{-d_{i}}v_{m+n})z^{j}\cdot T_{0}\\
=&(-1)^{p(-m-n)}Q_{0}^{-1}Q_{1}v_{m+n}z^{-j-1}+(Q_{1}-Q_{0}^{-1})v_{-m-n}\Big(\sum_{l=1}^{j}z^{j+1-2l}\Big)
+(Q_{0}^{-1}Q_{1}-1)v_{-m-n}\Big(\sum_{l=1}^{j+1}z^{j+2-2l}\Big)\\
&+\varsigma_{i}q^{-d_{i}}\Big[(-1)^{p(m+n)}v_{-m-n}z^{-j}+(Q_{1}-Q_{0}^{-1})v_{m+n}\Big(\sum_{l=1}^{j}z^{j-2l}\Big)
+(Q_{0}^{-1}Q_{1}-1)v_{m+n}\Big(\sum_{l=1}^{j}z^{j+1-2l}\Big)\Big],
\end{align*}
\begin{align*}
&B_{i}\cdot (v_{i}z^{j}\cdot T_{0})=B_{i}\cdot \Big[(-1)^{p(m+n+1)}v_{-m-n-1}z^{-j}+(Q_{1}-Q_{0}^{-1})v_{m+n+1}\Big(\sum_{l=1}^{j}z^{j-2l}\Big)\\
&+(Q_{0}^{-1}Q_{1}-1)v_{m+n+1}\Big(\sum_{l=1}^{j}z^{j+1-2l}\Big)\Big]\\
=&(-1)^{p(m+n+1)}(v_{-m-n}+\varsigma_{i}q^{-d_{i}}v_{-m-n-2})z^{-j}+(Q_{1}-Q_{0}^{-1})(v_{-m-n}z+\varsigma_{i}q^{-d_{i}}v_{m+n})\Big(\sum_{l=1}^{j}z^{j-2l}\Big)\\
&+(Q_{0}^{-1}Q_{1}-1)(v_{-m-n}z+\varsigma_{i}q^{-d_{i}}v_{m+n})\Big(\sum_{l=1}^{j}z^{j+1-2l}\Big).
\end{align*}
Since $q \varsigma_{m+n+1}=Q_0^{-1}Q_1$, we have the desired equality. The other cases can be shown similarly.

When $k=i=N-1$ and $j\geq 0$, we have
\begin{align*}
&(B_{N-1}\cdot v_{N-1}z^{j})\cdot T_{0}=v_{N}z^{j}\cdot T_{0}=v_{0}z^{j+1}\cdot T_{0}\\
=&Q_{0}^{-1}v_{0}z^{-j-1}+(Q_{1}-Q_{0}^{-1})v_{0}\Big(\sum_{l=1}^{j+1}z^{j+1-2l}\Big)+(Q_{0}^{-1}Q_{1}-1)v_{0}\Big(\sum_{l=1}^{j+1}z^{j+2-2l}\Big),
\end{align*}
\begin{align*}
&B_{N-1}\cdot (v_{N-1}z^{j}\cdot T_{0})=B_{N-1}\cdot \Big[ (-1)^{p(-1)}Q_{0}^{-1}Q_{1}v_{1}z^{-j-1}+(Q_{1}-Q_{0}^{-1})v_{-1}\Big(\sum_{l=1}^{j}z^{j+1-2l}\Big)\\
&+(Q_{0}^{-1}Q_{1}-1)v_{-1}\Big(\sum_{l=1}^{j+1}z^{j+2-2l}\Big)\Big]\\
=&(-1)^{p(-1)}Q_{0}^{-1}Q_{1}\varsigma_{N-1}v_{0}z^{-j-1}+(Q_{1}-Q_{0}^{-1})v_{0}\Big(\sum_{l=1}^{j}z^{j+1-2l}\Big)+(Q_{0}^{-1}Q_{1}-1)v_{0}\Big(\sum_{l=1}^{j+1}z^{j+2-2l}\Big).
\end{align*}
Since $Q_{0}^{-1}\varsigma_{N-1}=1$, we easily obtain $(B_{N-1}\cdot v_{N-1}z^{j})\cdot T_{0}=B_{N-1}\cdot (v_{N-1}z^{j}\cdot T_{0})$.

When $i=N-1$, $k=1$ and $j\geq 0$, we have
\begin{align*}
&(B_{N-1}\cdot v_{1}z^{j})\cdot T_{0}=\varsigma_{N-1}v_{0}z^{j}\cdot T_{0}\\
=&\varsigma_{N-1}\Big[ Q_{0}^{-1}v_{0}z^{-j}+(Q_{1}-Q_{0}^{-1})v_{0}\Big(\sum_{l=1}^{j}z^{j-2l}\Big)+(Q_{0}^{-1}Q_{1}-1)v_{0}\Big(\sum_{l=1}^{j}z^{j+1-2l}\Big)\Big],
\end{align*}
\begin{align*}
&B_{N-1}\cdot (v_{1}z^{j}\cdot T_{0})\\
=&B_{N-1}\cdot \Big[ (-1)^{p(1)}v_{-1}z^{-j}+(Q_{1}-Q_{0}^{-1})v_{1}\Big(\sum_{l=1}^{j}z^{j-2l}\Big)
+(Q_{0}^{-1}Q_{1}-1)v_{1}\Big(\sum_{l=1}^{j}z^{j+1-2l}\Big)\Big]\\
=&(-1)^{p(1)}v_{0}z^{-j}+(Q_{1}-Q_{0}^{-1})\varsigma_{N-1}v_{0}\Big(\sum_{l=1}^{j}z^{j-2l}\Big)+(Q_{0}^{-1}Q_{1}-1)\varsigma_{N-1}v_{0}\Big(\sum_{l=1}^{j}z^{j+1-2l}\Big).
\end{align*}
Since $Q_{0}^{-1}\varsigma_{N-1}=1$, we have the desired equality. The other cases can be proved similarly.
\end{proof}

\section{Schur--Weyl duality for $\widehat{U}^{\jmath}(\mathfrak{sl})$}\label{swrecio}
In this section we establish the double centralizer property between $\widehat{U}^{\jmath}(\mathfrak{sl})$ and $\mathbb{H}$ via their commutative actions on $\mathbb{V}^{\otimes d}$.
\subsection{Affine $\jmath$-Schur superalgebras}
From now on, we assume that $m+n\geq d$. Let $\Lambda_{m| n;d}$ be the set of weak compositions (with parities) of $d$ into $m+n+2$ parts, that is,
\begin{align}\label{ldabma}
\Lambda_{m| n;d} :=\Big\{\lambda=(\lambda^{(0)}| \lambda^{(1)})=(\lambda_{0}, \lambda_{1},\ldots,\lambda_{m}|\lambda_{m+1},\ldots,\lambda_{m+n+1})\in \mathbb{N}^{m+n+2}\Big|\sum_{i=0}^{m+n+1}\lambda_{i}=d\Big\}.
\end{align}

Let $W$ be the affine Weyl group of type $C_{d}$ generated by $S=\{s_{0}, s_{1},\ldots,s_{d}\}$ with the following affine Dynkin diagram:
\begin{center}
\begin{tikzpicture}
    \foreach \x/\label in {
        0/0, 1.5/1, 4.5/{d-1}, 6/d
    }{
        \draw (\x,0) circle (0.11);
        \node[below=4pt] at (\x,-0.11)
            {$\label$};
    }

    \node at (0.75,0) {$\Longrightarrow$};
    \node at (5.25,0) {$\Longleftarrow$};

    \draw (1.61,0) -- (2.45,0);
    \node at (3,0) {$\cdots$};
    \draw (3.55,0) -- (4.39,0);
\end{tikzpicture}
\end{center}
Let $W_{f}$ be the finite subgroup of $W$ generated by $s_{0}, s_{1},\ldots,s_{d-1}$, which is the finite Weyl group of type $C_{d}$. $W_{f}$ has a natural right action on $\mathbb{Z}^{d}$ as follows:
\begin{align*}
(f_1,f_2,\ldots,f_{d})\cdot s_{i}=\left\{
\begin{array}{ll}
(f_1,\ldots,f_{i-1},f_{i+1},f_{i},f_{i+2},\ldots,f_{d})&\hspace{0.35cm}\text{if $1\leq i\leq d-1$},\\[0.3em]
(-f_1,f_2,\ldots,f_{d})&\hspace{0.35cm}\text{if $i=0$}.
\end{array}
\right.
\end{align*}
Then by \cite[Prop. 4.2]{Hu90}, we have $W_{f}\ltimes \mathbb{Z}^{d}\cong W$. Under the identification, we have $(f_1,f_2,\ldots,f_{d})\cdot s_{d}=(f_1,f_2,\ldots,f_{d-1}, N-f_{d})$.

For $\lambda\in \Lambda_{m| n;d}$, let $W_{\lambda}$ be the finite parabolic subgroup of $W$ generated by $S\setminus \{s_{\widetilde{\lambda}_{0}},s_{\widetilde{\lambda}_{1}},\ldots,s_{\widetilde{\lambda}_{m+n}}\}$, where $\widetilde{\lambda}_{i}=\sum_{j=0}^{i}\lambda_{j}$ for $0\leq i\leq m+n$. Let $W_{\lambda^{(0)}}$ (resp. $W_{\lambda^{(1)}}$) be the subgroup of $W_{\lambda}$ generated by $\{s_{0},s_{1},\ldots,s_{\widetilde{\lambda}_{m}}\}\setminus \{s_{\widetilde{\lambda}_{0}},s_{\widetilde{\lambda}_{1}},\ldots,s_{\widetilde{\lambda}_{m}}\}$ (resp. $\{s_{\widetilde{\lambda}_{m}+1}, s_{\widetilde{\lambda}_{m}+2},\ldots,s_{d}\}\setminus \{s_{\widetilde{\lambda}_{m+1}},s_{\widetilde{\lambda}_{m+2}},\ldots,s_{\widetilde{\lambda}_{m+n}}\}$). Then we have $W_{\lambda}=W_{\lambda^{(0)}}\times W_{\lambda^{(1)}}$. Thus, for each $w\in W_{\lambda}$, we can uniquely write $w=w^{(0)}w^{(1)}$ for $w^{(0)}\in W_{\lambda^{(0)}}$ and $w^{(1)}\in W_{\lambda^{(1)}}$.

We define $T_{d}\in \mathbb{H}$ by
\begin{align*}
T_{d} :=Q_{0}^{-1}X_{d}T_{d-1}^{-1}\cdots T_{1}^{-1}T_{0}^{-1}T_{1}^{-1}\cdots T_{d-1}^{-1}.
\end{align*}
By \cite[Lemma 2.4]{FLLLWW20} and \cite[\S\S4.1-4.3]{St04}, we have the following result.
\begin{prop}\label{another}
$\mathbb{H}$ admits another presentation with generators $T_{0}, T_{1},\ldots, T_{d}$ and relations:
\begin{align*}
&(T_{0}-Q_{0}^{-1})(T_{0}+Q_{1})=0,\quad (T_{d}-Q_{1}^{-1})(T_{d}+Q_{0}^{-1})=0,\quad (T_{i}-q^{-1})(T_{i}+q)=0\quad \text{$(i\neq 0, d)$,}\\
&T_{k}T_{k-1}T_{k}=T_{k-1}T_{k}T_{k-1},\quad \text{$(2\leq k\leq d-1)$,}\\
&T_{0}T_{1}T_{0}T_{1}=T_{1}T_{0}T_{1}T_{0},\quad T_{d-1}T_{d}T_{d-1}T_{d}=T_{d}T_{d-1}T_{d}T_{d-1},\quad T_{i}T_{j}=T_{j}T_{i}\quad \text{$(|i-j|>1)$.}
\end{align*}
\end{prop}
For any $w\in W$ with a reduced expression $w=s_{i_{1}}\cdots s_{i_{t}}$, we set $T_{w} :=T_{i_{1}}\cdots T_{i_{t}}$ and $q_{w} :=q_{s_{i_{1}}}\cdots q_{s_{i_{t}}}$,
where
\begin{align*}
q_{s_{i}}=\left\{
\begin{array}{ll}
Q_{1}&\hspace{0.35cm}\text{if $i=0$ or $d$},\\[0.3em]
q&\hspace{0.35cm}\text{if $1\leq i\leq d-1$}.
\end{array}
\right.
\end{align*}
By Prop. \ref{another}, the above definitions are well-defined. Moreover, the set $\{T_{w}\:|\:w\in W\}$ forms an $\F$-basis of $\mathbb{H}$.

For $\lambda=(\lambda^{(0)}| \lambda^{(1)})\in \Lambda_{m| n;d}$, we set
\begin{align*}
x_{\lambda^{(0)}}=\sum_{w\in W_{\lambda^{(0)}}}q_{w}^{-1}T_{w},\qquad y_{\lambda^{(1)}}=\sum_{w\in W_{\lambda^{(1)}}}(-1)^{\ell(w)}q_{w}T_{w},
\end{align*}
where $\ell :W\rightarrow \mathbb{Z}_{\geq 0}$ is the length function on $W$. We have the following result.

\begin{lem}\label{lamadb}
For $\lambda\in \Lambda_{m| n;d}$ and for $i\in \{0, 1, \ldots, d\}\setminus \{\widetilde{\lambda}_{0},\widetilde{\lambda}_{1},\ldots,\widetilde{\lambda}_{m+n}\}$, we have
\begin{align*}
x_{\lambda^{(0)}}y_{\lambda^{(1)}}T_{i}=\left\{
\begin{array}{ll}
Q_{0}^{-1}x_{\lambda^{(0)}}y_{\lambda^{(1)}}&\hspace{0.35cm}\text{if $i=0$},\\[0.3em]
q^{-1}x_{\lambda^{(0)}}y_{\lambda^{(1)}}&\hspace{0.35cm}\text{if $1\leq i< \widetilde{\lambda}_{m}$},\\[0.3em]
-qx_{\lambda^{(0)}}y_{\lambda^{(1)}}&\hspace{0.35cm}\text{if $\widetilde{\lambda}_{m}< i\leq d-1$},\\[0.3em]
-Q_{0}^{-1}x_{\lambda^{(0)}}y_{\lambda^{(1)}}&\hspace{0.35cm}\text{if $i=d$}.
\end{array}
\right.
\end{align*}
\end{lem}
\begin{proof}
We have $x_{\lambda^{(0)}}=\sum_{w\in W_{\lambda^{(0)}}\atop ws_{i}< w}q_{w}^{-1}(T_{w}+q_{s_{i}}T_{ws_{i}})$ when $s_{i}\in W_{\lambda^{(0)}}$, and $y_{\lambda^{(1)}}=\sum_{w\in W_{\lambda^{(1)}}\atop ws_{i}< w}(-1)^{\ell(w)}q_{w}(T_{w}+(-q_{s_{i}}^{-1})T_{ws_{i}})$ when $s_{i}\in W_{\lambda^{(1)}}$, where $<$ is the Bruhat ordering on $W$. Then
\begin{align*}
x_{\lambda^{(0)}}y_{\lambda^{(1)}}T_{0}=x_{\lambda^{(0)}}T_{0}y_{\lambda^{(1)}}=\sum_{w\in W_{\lambda^{(0)}}\atop ws_{0}< w}q_{w}^{-1}T_{ws_{0}}(T_{0}+q_{s_{0}})T_{0}y_{\lambda^{(1)}}=Q_{0}^{-1}x_{\lambda^{(0)}}y_{\lambda^{(1)}}.
\end{align*}
\begin{align*}
x_{\lambda^{(0)}}y_{\lambda^{(1)}}T_{i}=x_{\lambda^{(0)}}T_{i}y_{\lambda^{(1)}}=\sum_{w\in W_{\lambda^{(0)}}\atop ws_{i}< w}q_{w}^{-1}T_{ws_{i}}(T_{i}+q_{s_{i}})T_{i}y_{\lambda^{(1)}}=q^{-1}x_{\lambda^{(0)}}y_{\lambda^{(1)}}\quad \text{for $1\leq i< \widetilde{\lambda}_{m}$}.
\end{align*}
\begin{align*}
x_{\lambda^{(0)}}y_{\lambda^{(1)}}T_{i}=x_{\lambda^{(0)}}\sum_{w\in W_{\lambda^{(1)}}\atop ws_{i}< w}(-1)^{\ell(w)}q_{w}T_{ws_{i}}(T_{i}+(-q_{s_{i}}^{-1}))T_{i}=-qx_{\lambda^{(0)}}y_{\lambda^{(1)}}\quad \text{for $\widetilde{\lambda}_{m}< i\leq d-1$}.
\end{align*}
\begin{align*}
x_{\lambda^{(0)}}y_{\lambda^{(1)}}T_{d}=x_{\lambda^{(0)}}\sum_{w\in W_{\lambda^{(1)}}\atop ws_{d}< w}(-1)^{\ell(w)}q_{w}T_{ws_{d}}(T_{d}+(-q_{s_{d}}^{-1}))T_{d}=-Q_{0}^{-1}x_{\lambda^{(0)}}y_{\lambda^{(1)}}.
\end{align*}
\end{proof}

For each $\lambda\in \Lambda_{m| n;d}$, we define $$p(\lambda)=\sum\limits_{i=m+1}^{m+n+1}\lambda_{i}\quad (\text{mod }2).$$
Now we can define the affine $\jmath$-Schur superalgebra $\mathbb{S}_{m| n;d}^{\jmath}$ of three parameters as follows:
\begin{align*}
\mathbb{S}_{m| n;d}^{\jmath} :=\mathrm{End}_{\mathbb{H}}(\oplus_{\lambda\in \Lambda_{m| n;d}}x_{\lambda^{(0)}}y_{\lambda^{(1)}}\mathbb{H})=\bigoplus_{\lambda, \mu\in \Lambda_{m| n;d}}\mathrm{Hom}_{\mathbb{H}}(x_{\mu^{(0)}}y_{\mu^{(1)}}\mathbb{H}, x_{\lambda^{(0)}}y_{\lambda^{(1)}}\mathbb{H}),
\end{align*}
together with a $\Z_2$-grading defined by
$$\left(\mathbb{S}_{m| n;d}^{\jmath}\right)_{\overline{a}}=\bigoplus_{\substack{\lambda, \mu\in \Lambda_{m| n;d}\\p(\lambda)+p(\mu)=\overline{a}}}\mathrm{Hom}_{\mathbb{H}}(x_{\mu^{(0)}}y_{\mu^{(1)}}\mathbb{H}, x_{\lambda^{(0)}}y_{\lambda^{(1)}}\mathbb{H}),\quad  \text{for }\overline{a}\in \Z_2.$$

Let $\mathcal{D}_{\lambda}$ (resp. $\mathcal{D}_{\lambda}^{-1}$) be the minimal length right (resp. left) coset representatives of $W_{\lambda}$ in $W$, that is,
\begin{align*}
\mathcal{D}_{\lambda}=\{g\in W\:|\:\ell(wg)=\ell(w)+\ell(g),~\forall w\in W_{\lambda}\}.
\end{align*}
Denote by $\mathcal{D}_{\lambda\mu}=\mathcal{D}_{\lambda}\cap \mathcal{D}_{\mu}^{-1}$ the set of minimal length double coset representatives of $W_{\lambda}\backslash W/W_{\mu}$.

For $\lambda, \mu\in \Lambda_{m| n;d}$ and $g\in \mathcal{D}_{\lambda\mu}$, by \cite[Prop. 2.4]{FLLLW23} there exists a weak composition $\delta \in \Lambda_{m_{1}| n_{1};d}$ for some $m_{1}, n_{1}$ such that $W_{\delta} =g^{-1}W_{\lambda}g\cap W_{\mu}$. Besides, the map $W_{\lambda} \times (\mathcal{D}_{\delta}\cap W_{\mu})\rightarrow W_{\lambda}gW_{\mu}$ sending $(x, y)$ to $xgy$ is a bijection, and moreover, we have $\ell(xgy)=\ell(x)+\ell(g)+\ell(y)$; the map $W_{\delta} \times (\mathcal{D}_{\delta}\cap W_{\mu})\rightarrow W_{\mu}$ sending $(x, y)$ to $xy$ is a bijection, and moreover, we have $\ell(xy)=\ell(x)+\ell(y)$. Set
\begin{align*}
\mathcal{D}_{\lambda\mu}^{\circ} :=\big\{g\in \mathcal{D}_{\lambda\mu}\:|\:g^{-1}W_{\lambda^{(1)}}g\cap W_{\mu^{(0)}}=g^{-1}W_{\lambda^{(0)}}g\cap W_{\mu^{(1)}}=\{e\}\big\}.
\end{align*}
If $g\in \mathcal{D}_{\lambda\mu}^{\circ}$, we see that $W_{\delta} =W_{\delta^{(0)}}\times W_{\delta^{(1)}}$, where $W_{\delta^{(0)}}=g^{-1}W_{\lambda^{(0)}}g\cap W_{\mu^{(0)}}$ and $W_{\delta^{(1)}}=g^{-1}W_{\lambda^{(1)}}g\cap W_{\mu^{(1)}}$.

For $\lambda, \mu\in \Lambda_{m| n;d}$ and $g\in \mathcal{D}_{\lambda\mu}^{\circ}$, we define $\phi_{\lambda, \mu}^{g}\in \mathbb{S}_{m| n;d}^{\jmath}$ by
\begin{align*}
\phi_{\lambda, \mu}^{g}(x_{\nu^{(0)}}y_{\nu^{(1)}}h)=\delta_{\mu, \nu}T_{W_{\lambda}gW_{\mu}}h,\quad \forall \nu\in \Lambda_{m| n;d} \text{ and }h\in \mathbb{H},
\end{align*}
where $T_{W_{\lambda}gW_{\mu}}=x_{\lambda^{(0)}}y_{\lambda^{(1)}}T_{g}T_{\mathcal{D}_{\delta}\cap W_{\mu}}$ and
\begin{align*}
T_{\mathcal{D}_{\delta}\cap W_{\mu}}=\sum\limits_{w=w^{(0)}w^{(1)}\in \mathcal{D}_{\delta}\cap W_{\mu}}(-1)^{\ell(w^{(1)})}q_{w^{(0)}}^{-1}q_{w^{(1)}}T_{w}.
\end{align*}

For any $s_{i}\in W_{\mu}$, by Lemma \ref{lamadb} we have $x_{\mu^{(0)}}y_{\mu^{(1)}}T_{s_{i}}= A x_{\mu^{(0)}}y_{\mu^{(1)}}$ for some $A\in \F$. Then we have
\begin{align*}
T_{W_{\lambda}gW_{\mu}}T_{s_{i}}= A\cdot T_{W_{\lambda}gW_{\mu}},
\end{align*}
which follows from the fact that $T_{W_{\lambda}gW_{\mu}}=h\cdot x_{\mu^{(0)}}y_{\mu^{(1)}}$ for some $h\in \mathbb{H}$. Thus $\phi_{\lambda, \mu}^{g}$ is a well-defined right $\mathbb{H}$-module homomorphism from $x_{\mu^{(0)}}y_{\mu^{(1)}}\mathbb{H}$ to $x_{\lambda^{(0)}}y_{\lambda^{(1)}}\mathbb{H}$. By a standard argument we have the following lemma (cf. \cite{DR11} and \cite{CL23} for super finite types $A$ and $B$, and also \cite{VV99}, \cite{DDF12}, \cite{FLLLW23} for affine types $A$ and $B$).
\begin{lem}
The set $\{\phi_{\lambda, \mu}^{g}\:|\:\lambda, \mu\in \Lambda_{m| n;d}, ~g\in \mathcal{D}_{\lambda\mu}^{\circ}\}$ forms an $\F$-basis of $\mathbb{S}_{m| n;d}^{\jmath}$.
\end{lem}

Define the right $\mathbb{H}$-module
\begin{align*}
\mathbb{T}_{m| n;d} :=\bigoplus_{\lambda\in \Lambda_{m| n;d}}x_{\lambda^{(0)}}y_{\lambda^{(1)}}\mathbb{H},
\end{align*}
whose $\Z_2$-grading is
$$(\mathbb{T}_{m| n;d})_{\overline{a}}=\bigoplus_{\substack{\lambda\in \Lambda_{m| n;d}\\p(\lambda)=\overline{a}}}x_{\lambda^{(0)}}y_{\lambda^{(1)}}\mathbb{H},\quad  \text{for }\overline{a}\in \Z_2.$$
By Lemma \ref{lamadb} we can easily obtain the following result.
\begin{lem}\label{mamnbb}
$\mathbb{T}_{m| n;d}$, as an $\mathbb{H}$-module, has an $\F$-basis
\begin{align*}
\{x_{\lambda^{(0)}}y_{\lambda^{(1)}}T_{g}\:|\:\lambda\in \Lambda_{m| n;d}, ~g\in \mathcal{D}_{\lambda}\},
\end{align*}
\end{lem}

\begin{lem}\label{votthma}
The tensor superspace $\mathbb{V}^{\otimes d}$ is generated by $\{M_{f}\:|\:0\leq f_{1}\leq f_{2}\leq \cdots \leq f_{d}\leq m+n+1\}$ as an $\mathbb{H}$-module. More precisely, $\mathbb{V}^{\otimes d}$ has an $\F$-basis
\begin{align*}
\{M_{f^{\lambda}}\cdot T_{g}\:|\:\lambda\in \Lambda_{m| n;d}, ~g\in \mathcal{D}_{\lambda}\},
\end{align*}
where $$f^{\lambda} =\{\underbrace{0,\ldots,0}_{\lambda_{0}},\underbrace{1,\ldots,1}_{\lambda_{1}},\ldots,\underbrace{m,\ldots,m}_{\lambda_{m}},\underbrace{m+1,\ldots,m+1}_{\lambda_{m+1}},\ldots,\underbrace{m+n+1,\ldots,m+n+1}_{\lambda_{m+n+1}}\}$$ if $\lambda=(\lambda_{0}, \lambda_{1},\ldots,\lambda_{m}|\lambda_{m+1},\ldots,\lambda_{m+n+1})$.
\end{lem}
\begin{proof}
Recall the natural right action of $W$ on $\mathbb{Z}^{d}$. From this we have
$$ \mathbb Z^d=\bigsqcup_{\lambda\in\Lambda_{m|n;d}} f^\lambda W. $$
Moreover, the stabilizer of \(f^\lambda\) in $W$ is exactly \(W_\lambda\). Thus, the set $$ \{M_{f^\lambda g}\mid \lambda\in\Lambda_{m|n;d},\ g\in D_\lambda\} $$ forms an $\F$-basis of $\mathbb{V}^{\otimes d}$. Next, for each $g\in \mathcal{D}_{\lambda}$, by induction on its length we can obtain
$$ M_{f^{\lambda}}\cdot T_g = c_{\lambda,g}M_{f^\lambda g} +\sum_{\substack{h\in D_\lambda\\h<g}} c_{\lambda,h}M_{f^\lambda h}, \quad \text{where }c_{\lambda,g}\neq0.$$
From this, we can easily obtain the desired result.
\end{proof}

Summarizing the above, we obtain the following result.
\begin{lem}\label{mathtmnd}
There exists a unique grading-preserving $\mathbb{H}$-module isomorphism $\kappa :\mathbb{T}_{m| n;d}\rightarrow \mathbb{V}^{\otimes d}$ which sends $x_{\lambda^{(0)}}y_{\lambda^{(1)}}h$ to $M_{f^{\lambda}}\cdot h$ for any $\lambda\in \Lambda_{m| n;d}$ and $h\in \mathbb{H}$.
\end{lem}
\begin{proof}
For any $s_{i}\in W_{\lambda}$, by Lemma \ref{lamadb} we have $x_{\lambda^{(0)}}y_{\lambda^{(1)}}T_{s_{i}}= Ax_{\lambda^{(0)}}y_{\lambda^{(1)}}$ for $A=Q_{0}^{-1}, q^{-1}, -q$ or $-Q_{0}^{-1}$. We need to show that $M_{f^{\lambda}}\cdot T_{s_{i}} =A M_{f^{\lambda}}$. The case when $i\neq d$ is easy. When $i=d$, we need to show
\begin{align*}
M_{f^{\lambda}}\cdot X_{d}=-M_{f^{\lambda}}\cdot T_{d-1}\cdots T_{1}T_{0}T_{1}\cdots T_{d-1}.
\end{align*}
We have
\begin{align*}
M_{f^{\lambda}}\cdot (T_{d-1}\cdots T_{1})=(-q)^{\lambda_{m+n+1}-1}(-1)^{\sum_{i=m+1}^{m+n}\lambda_{i}}v_{m+n+1}\otimes v_{0}^{\otimes \lambda_{0}}\otimes v_{1}^{\otimes \lambda_{1}}\otimes \cdots \otimes v_{m+n+1}^{\otimes (\lambda_{m+n+1}-1)},
\end{align*}
\begin{align*}
v_{m+n+1}\otimes v_{0}^{\otimes \lambda_{0}}\otimes v_{1}^{\otimes \lambda_{1}}\otimes \cdots \otimes v_{m+n+1}^{\otimes \lambda_{m+n+1}-1}\cdot T_{0}=-(v_{m+n+1}\otimes v_{0}^{\otimes \lambda_{0}}\otimes v_{1}^{\otimes \lambda_{1}}\otimes \cdots \otimes v_{m+n+1}^{\otimes (\lambda_{m+n+1}-1)})z_{1}^{-1},
\end{align*}
\begin{align*}
&(v_{m+n+1}\otimes v_{0}^{\otimes \lambda_{0}}\otimes v_{1}^{\otimes \lambda_{1}}\otimes \cdots \otimes v_{m+n+1}^{\otimes (\lambda_{m+n+1}-1)})z_{1}^{-1}\cdot (T_{1}\cdots T_{d-\lambda_{m+n+1}})\\
=&(-1)^{\sum_{i=m+1}^{m+n}\lambda_{i}}(v_{0}^{\otimes \lambda_{0}}\otimes v_{1}^{\otimes \lambda_{1}}\otimes \cdots \otimes v_{m+n+1}^{\otimes \lambda_{m+n+1}})z_{d-\lambda_{m+n+1}+1}^{-1},
\end{align*}
\begin{align*}
&(v_{0}^{\otimes \lambda_{0}}\otimes v_{1}^{\otimes \lambda_{1}}\otimes \cdots \otimes v_{m+n+1}^{\otimes \lambda_{m+n+1}})z_{d-\lambda_{m+n+1}+1}^{-1}\cdot  (T_{d-\lambda_{m+n+1}+1}\cdots T_{d-1})\\
=&(-q^{-1})^{\lambda_{m+n+1}-1}(v_{0}^{\otimes \lambda_{0}}\otimes v_{1}^{\otimes \lambda_{1}}\otimes \cdots \otimes v_{m+n+1}^{\otimes \lambda_{m+n+1}})z_{d}^{-1}.
\end{align*}
Thus, we obtain the desired equality. Therefore, the morphism $\kappa$, sending $x_{\lambda^{(0)}}y_{\lambda^{(1)}}h$ to $M_{f^{\lambda}}\cdot h$, is a well-defined right $\mathbb{H}$-module homomorphism.
Combining Lemma \ref{mamnbb} with Lemma \ref{votthma}, we see that $\kappa$ is an isomorphism.
\end{proof}

\subsection{Schur--Weyl duality}
In this subsection we adhere to the assumption of \S\ref{comm} such that Theorem \ref{hatcuth} holds. By Lemma \ref{mathtmnd}, we have an isomorphism of superalgebras $\mathbb{S}_{m| n;d}^{\jmath} \cong \mathrm{End}_{\mathbb{H}}(\mathbb{V}^{\otimes d})$. By this identification and Theorem \ref{hatcuth}, we obtain an $\F$-algebra homomorphism
\begin{align*}
\Psi :\widehat{U}^{\jmath}(\mathfrak{sl})\rightarrow \mathbb{S}_{m| n;d}^{\jmath}.
\end{align*}
Next we want to show that $\Psi$ is surjective.

We define two families of maps $\widetilde{e}_{i}, \widetilde{f}_{i} :\Lambda_{m| n;d}\rightarrow \Lambda_{m| n;d}\sqcup \{0\}$, where $0$ is a formal symbol, by
\begin{align*}
\widetilde{e}_{i}(\lambda)=\left\{
\begin{array}{ll}
(\lambda_{0}, \ldots,\lambda_{i-1},\lambda_{i}+1,\lambda_{i+1}-1,\lambda_{i+2},\ldots,\lambda_{m+n+1})&\hspace{0.35cm}\text{if $\lambda_{i+1}> 0$},\\[0.3em]
0&\hspace{0.35cm}\text{if $\lambda_{i+1}= 0$};
\end{array}
\right.
\end{align*}
\begin{align*}
\widetilde{f}_{i}(\lambda)=\left\{
\begin{array}{ll}
(\lambda_{0}, \ldots,\lambda_{i-1},\lambda_{i}-1,\lambda_{i+1}+1,\lambda_{i+2},\ldots,\lambda_{m+n+1})&\hspace{0.35cm}\text{if $\lambda_{i}> 0$},\\[0.3em]
0&\hspace{0.35cm}\text{if $\lambda_{i}= 0$}.
\end{array}
\right.
\end{align*}
It is understood that $M_{f^{0}}=\{0\}$ and $\phi_{0, \mu}^{g}=0=\phi_{\lambda, 0}^{g}$. Then we have the following result.

\begin{lem}\label{bibmn}
We have
\begin{align*}
\Psi(B_{i})=\sum\limits_{\lambda\in \Lambda_{m| n;d}}q^{d_{i+1}\lambda_{i+1}}\phi_{\widetilde{f}_{i}(\lambda), \lambda}^{e},\qquad \text{for $1\leq i\leq m+n$},
\end{align*}
\begin{align*}
\Psi(B_{i})=\sum\limits_{\lambda\in \Lambda_{m| n;d}}\varsigma_{i}q^{d_{\tau i}\lambda_{\tau i}}\phi_{\widetilde{e}_{\tau i}(\lambda), \lambda}^{e},\qquad \text{for $m+n+2\leq i\leq 2m+2n+1$},
\end{align*}
\begin{align*}
\Psi(B_{0})=\sum\limits_{\lambda\in \Lambda_{m| n;d}}q^{\lambda_{1}}\phi_{\widetilde{f}_{0}(\lambda), \lambda}^{e},
\end{align*}
\begin{align*}
\Psi(B_{m+n+1})=\sum\limits_{\lambda\in \Lambda_{m| n;d}}q^{-\lambda_{m+n}-\lambda_{m+n+1}+1}\phi_{\widetilde{e}_{m+n}(\lambda), \lambda}^{e}.
\end{align*}
\end{lem}
\begin{proof}
Assume that $1\leq i\leq m+n$. By \eqref{bimf}, we have
\begin{align*}
\Psi(B_{i})M_{f^{\lambda}}=&\sum\limits_{1\leq a\leq d\atop f^{\lambda}_{a}\equiv i}\epsilon_{i, a}(f^{\lambda})\Big(\prod_{r> a}\kappa_{i}(f^{\lambda}_{r})^{-1}\Big)M_{f^{\lambda}+e_{a}}\\
=&\sum\limits_{t=1}^{\lambda_{i}}(-1)^{|\alpha_i|\sum\limits_{s< \widetilde{\lambda}_{i-1}+t}
p(f^{\lambda}_{s})}q^{-d_{i}(\lambda_{i}-t)+d_{i+1}\lambda_{i+1}}\\
&\times M_{(0^{\lambda_{0}},\ldots,(i-1)^{\lambda_{i-1}},i^{t-1},i+1,i^{\lambda_{i}-t},(i+1)^{\lambda_{i+1}},\ldots,(m+n+1)^{\lambda_{m+n+1}})}.
\end{align*}

When $1\leq i< m$, we have $d_{i}=d_{i+1}=1$ and $|\alpha_i|=\overline{0}$. Thus, in this case, $\Psi(B_{i})M_{f^{\lambda}}=\sum\limits_{t=1}^{\lambda_{i}}q^{-(\lambda_{i}-t)+\lambda_{i+1}}M_{f^{\lambda}+e_{\widetilde{\lambda}_{i-1}+t}}$. On the other hand,
\begin{align*}
\phi_{\widetilde{f}_{i}(\lambda), \lambda}^{e}(M_{f^{\lambda}})=\sum\limits_{w=w^{(0)}w^{(1)}\in \mathcal{D}_{\widetilde{f}_{i}(\lambda)}\cap W_{\lambda}}(-1)^{\ell(w^{(1)})}q_{w^{(0)}}^{-1}q_{w^{(1)}}M_{\widetilde{f}_{i}(\lambda)}T_{w}.
\end{align*}
Note that $\mathcal{D}_{\widetilde{f}_{i}(\lambda)}\cap W_{\lambda}=\big\{s_{\widetilde{\lambda}_{i}-1}s_{\widetilde{\lambda}_{i}-2}\cdots s_{\widetilde{\lambda}_{i-1}+t}\big\}^{\lambda_{i}}_{t=1}$. Hence, we have
\begin{align*}
\phi_{\widetilde{f}_{i}(\lambda), \lambda}^{e}(M_{f^{\lambda}})=\sum\limits_{t=1}^{\lambda_{i}}q^{-(\lambda_{i}-t)}M_{\widetilde{f}_{i}(\lambda)}T_{\widetilde{\lambda}_{i}-1}T_{\widetilde{\lambda}_{i}-2}\cdots T_{\widetilde{\lambda}_{i-1}+t}=\sum\limits_{t=1}^{\lambda_{i}}q^{-(\lambda_{i}-t)}M_{f^{\lambda}+e_{\widetilde{\lambda}_{i-1}+t}}.
\end{align*}
Thus, we obtain, when $1\leq i< m$,
\begin{align*}
\Psi(B_{i})=\sum\limits_{\lambda\in \Lambda_{m| n;d}}q^{\lambda_{i+1}}\phi_{\widetilde{f}_{i}(\lambda), \lambda}^{e}.
\end{align*}

When $i=m$, we have $d_{m}=1$, $d_{m+1}=-1$, $|\alpha_m|=\overline{1}$ and $p(f^{\lambda}_{s})=\overline{0}$ for any $s< \widetilde{\lambda}_{m}$. Thus, we have
$\Psi(B_{m})M_{f^{\lambda}}=\sum\limits_{t=1}^{\lambda_{m}}q^{-(\lambda_{m}-t)-\lambda_{m+1}}M_{f^{\lambda}+e_{\widetilde{\lambda}_{i-1}+t}}$. Moreover, we have
\begin{align*}
\phi_{\widetilde{f}_{m}(\lambda), \lambda}^{e}(M_{f^{\lambda}})=\sum\limits_{t=1}^{\lambda_{m}}q^{-(\lambda_{m}-t)}M_{\widetilde{f}_{m}(\lambda)}T_{\widetilde{\lambda}_{m}-1}T_{\widetilde{\lambda}_{m}-2}\cdots T_{\widetilde{\lambda}_{m-1}+t}=\sum\limits_{t=1}^{\lambda_{m}}q^{-(\lambda_{m}-t)}M_{f^{\lambda}+e_{\widetilde{\lambda}_{m-1}+t}}.
\end{align*}
Hence we obtain
\begin{align*}
\Psi(B_{m})=\sum\limits_{\lambda\in \Lambda_{m| n;d}}q^{-\lambda_{m+1}}\phi_{\widetilde{f}_{m}(\lambda), \lambda}^{e}.
\end{align*}

When $m+1\leq i\leq m+n$, we have $d_{i}=d_{i+1}=-1$ and $|\alpha_i|=\overline{0}$. Thus, in this case, $\Psi(B_{i})M_{f^{\lambda}}=\sum\limits_{t=1}^{\lambda_{i}}q^{(\lambda_{i}-t)-\lambda_{i+1}}M_{f^{\lambda}+e_{\widetilde{\lambda}_{i-1}+t}}$. Meanwhile, we have
\begin{align*}
\phi_{\widetilde{f}_{i}(\lambda), \lambda}^{e}(M_{f^{\lambda}})=\sum\limits_{t=1}^{\lambda_{i}}(-1)^{\lambda_{i}-t}q^{\lambda_{i}-t}M_{\widetilde{f}_{i}(\lambda)}T_{\widetilde{\lambda}_{i}-1}T_{\widetilde{\lambda}_{i}-2}\cdots T_{\widetilde{\lambda}_{i-1}+t}=\sum\limits_{t=1}^{\lambda_{i}}q^{\lambda_{i}-t}M_{f^{\lambda}+e_{\widetilde{\lambda}_{i-1}+t}}.
\end{align*}
Hence, when $m+1\leq i\leq m+n$ we have
\begin{align*}
\Psi(B_{i})=\sum\limits_{\lambda\in \Lambda_{m| n;d}}q^{-\lambda_{i+1}}\phi_{\widetilde{f}_{i}(\lambda), \lambda}^{e}.
\end{align*}

Assume that $m+n+2\leq i\leq 2m+2n+1$. By \eqref{bimf}, we have
\begin{align*}
\Psi(B_{i})M_{f^{\lambda}}=&\varsigma_{i}\sum\limits_{1\leq a\leq d\atop f^{\lambda}_{a}\equiv -i}\epsilon_{i, a}(f^{\lambda})\Big(\prod_{r< a}\kappa_{i}(f^{\lambda}_{r})^{-1}\kappa_{\tau i}(f^{\lambda}_{r})\Big)\kappa_{i}(f^{\lambda}_{a})^{-1}\Big(\prod_{r> a}\kappa_{i}(f^{\lambda}_{r})^{-1}\Big)M_{f^{\lambda}-e_{a}}\\
=&\varsigma_{i}\sum\limits_{t=1}^{\lambda_{N-i}}(-1)^{|\alpha_i|\sum\limits_{s< \widetilde{\lambda}_{N-i-1}+t}
p(f^{\lambda}_{s})}q^{d_{N-i-1}\lambda_{N-i-1}-d_{N-i}(t-1)}\\
&\times M_{(0^{\lambda_{0}},\ldots,(N-i-1)^{\lambda_{N-i-1}},(N-i)^{t-1},N-i-1,(N-i)^{\lambda_{N-i}-t},(N-i+1)^{\lambda_{N-i+1}},\ldots,(m+n+1)^{\lambda_{m+n+1}})}.
\end{align*}

When $m+n+2\leq i\leq m+2n$, we have $d_{N-i-1}=d_{N-i}=-1$ and $|\alpha_i|=\overline{0}$. Thus, in this case, $\Psi(B_{i})M_{f^{\lambda}}=\varsigma_{i}\sum\limits_{t=1}^{\lambda_{N-i}}q^{-\lambda_{N-i-1}+(t-1)}M_{f^{\lambda}-e_{\widetilde{\lambda}_{N-i-1}+t}}$. On the other hand,
\begin{align*}
\phi_{\widetilde{e}_{N-i-1}(\lambda), \lambda}^{e}(M_{f^{\lambda}})=\sum\limits_{w=w^{(0)}w^{(1)}\in \mathcal{D}_{\widetilde{e}_{N-i-1}(\lambda)}\cap W_{\lambda}}(-1)^{\ell(w^{(1)})}q_{w^{(0)}}^{-1}q_{w^{(1)}}M_{\widetilde{e}_{N-i-1}(\lambda)}T_{w}.
\end{align*}
Note that $\mathcal{D}_{\widetilde{e}_{N-i-1}(\lambda)}\cap W_{\lambda}=\big\{s_{\widetilde{\lambda}_{N-i-1}+1}s_{\widetilde{\lambda}_{N-i-1}+2}\cdots s_{\widetilde{\lambda}_{N-i-1}+t-1}\big\}^{\lambda_{N-i}}_{t=1}$. Hence, we have
\begin{align*}
\phi_{\widetilde{e}_{N-i-1}(\lambda), \lambda}^{e}(M_{f^{\lambda}})=&\sum\limits_{t=1}^{\lambda_{N-i}}(-1)^{t-1}q^{t-1}M_{\widetilde{e}_{N-i-1}(\lambda)}T_{\widetilde{\lambda}_{N-i-1}+1}T_{\widetilde{\lambda}_{N-i-1}+2}\cdots T_{\widetilde{\lambda}_{N-i-1}+t-1}\\
=&\sum\limits_{t=1}^{\lambda_{N-i}}q^{t-1}M_{f^{\lambda}-e_{\widetilde{\lambda}_{N-i-1}+t}}.
\end{align*}
Thus, we obtain, when $m+n+2\leq i\leq m+2n$,
\begin{align*}
\Psi(B_{i})=\sum\limits_{\lambda\in \Lambda_{m| n;d}}\varsigma_{i}q^{-\lambda_{N-i-1}}\phi_{\widetilde{e}_{N-i-1}(\lambda), \lambda}^{e}.
\end{align*}

When $i=m+2n+1$, that is, $N-i-1=m$, we have $d_{N-i-1}=1$, $d_{N-i}=-1$ and $|\alpha_i|=\overline{1}$. Thus, we have $\Psi(B_{i})M_{f^{\lambda}}=\varsigma_{i}\sum\limits_{t=1}^{\lambda_{N-i}}(-1)^{\sum\limits_{s< \widetilde{\lambda}_{m}+t}p(f^{\lambda}_{s})}q^{\lambda_{N-i-1}+(t-1)}M_{f^{\lambda}-e_{\widetilde{\lambda}_{m}+t}}=
\varsigma_{i}\sum\limits_{t=1}^{\lambda_{N-i}}(-1)^{t-1}q^{\lambda_{m}+(t-1)}M_{f^{\lambda}-e_{\widetilde{\lambda}_{m}+t}}$. Meanwhile, we have
\begin{align*}
\phi_{\widetilde{e}_{m}(\lambda), \lambda}^{e}(M_{f^{\lambda}})=&\sum\limits_{t=1}^{\lambda_{N-i}}(-1)^{t-1}q^{t-1}M_{\widetilde{e}_{m}(\lambda)}T_{\widetilde{\lambda}_{m}+1}T_{\widetilde{\lambda}_{m}+2}\cdots T_{\widetilde{\lambda}_{m}+t-1}\\
=&\sum\limits_{t=1}^{\lambda_{N-i}}(-1)^{t-1}q^{t-1}M_{f^{\lambda}-e_{\widetilde{\lambda}_{m}+t}}.
\end{align*}
Thus, when $i=m+2n+1$ we have
\begin{align*}
\Psi(B_{i})=\sum\limits_{\lambda\in \Lambda_{m| n;d}}\varsigma_{i}q^{\lambda_{m}}\phi_{\widetilde{e}_{m}(\lambda), \lambda}^{e}.
\end{align*}

When $m+2n+2\leq i\leq 2m+2n+1$, we have $d_{N-i-1}=d_{N-i}=1$ and $|\alpha_i|=\overline{0}$. Thus, in this case,
$\Psi(B_{i})M_{f^{\lambda}}=\varsigma_{i}\sum\limits_{t=1}^{\lambda_{N-i}}q^{\lambda_{N-i-1}-(t-1)}M_{f^{\lambda}-e_{\widetilde{\lambda}_{N-i-1}+t}}$. Moreover, we have
\begin{align*}
\phi_{\widetilde{e}_{N-i-1}(\lambda), \lambda}^{e}(M_{f^{\lambda}})=&\sum\limits_{t=1}^{\lambda_{N-i}}q^{-(t-1)}M_{\widetilde{e}_{N-i-1}(\lambda)}T_{\widetilde{\lambda}_{N-i-1}+1}T_{\widetilde{\lambda}_{N-i-1}+2}\cdots T_{\widetilde{\lambda}_{N-i-1}+t-1}\\
=&\sum\limits_{t=1}^{\lambda_{N-i}}q^{-(t-1)}M_{f^{\lambda}-e_{\widetilde{\lambda}_{N-i-1}+t}}.
\end{align*}
Hence, when $m+2n+2\leq i\leq 2m+2n+1$, we have
\begin{align*}
\Psi(B_{i})=\sum\limits_{\lambda\in \Lambda_{m| n;d}}\varsigma_{i}q^{\lambda_{N-i-1}}\phi_{\widetilde{e}_{N-i-1}(\lambda), \lambda}^{e}.
\end{align*}

When $i=0$, we have $d_{0}=d_{1}=1$ and $|\alpha_0|=\overline{0}$. By \eqref{bimf}, we have
\begin{align*}
\Psi(B_{0})M_{f^{\lambda}}=&\sum\limits_{1\leq a\leq d\atop f^{\lambda}_{a}\equiv 0}\epsilon_{0, a}(f^{\lambda})\Big(\prod_{r> a}\kappa_{0}(f^{\lambda}_{r})^{-1}\Big)M_{f^{\lambda}+e_{a}}\\
&\notag+\varsigma_{0}\sum\limits_{1\leq a\leq d\atop f^{\lambda}_{a}\equiv 0}\epsilon_{0, a}(f^{\lambda})\Big(\prod_{r< a}\kappa_{0}(f^{\lambda}_{r})^{-1}\kappa_{N-1}(f^{\lambda}_{r})\Big)\kappa_{0}(f^{\lambda}_{a})^{-1}\Big(\prod_{r> a}\kappa_{0}(f^{\lambda}_{r})^{-1}\Big)M_{f^{\lambda}-e_{a}}\\
=&\sum\limits_{t=1}^{\lambda_{0}}q^{-(\lambda_{0}-t)+\lambda_{1}}M_{f^{\lambda}+e_{t}}+\varsigma_{0}\sum\limits_{t=1}^{\lambda_{0}}q^{-\lambda_{0}+\lambda_{1}-(t-1)}M_{f^{\lambda}-e_{t}}.
\end{align*}
On the other side, we have
\begin{align*}
\phi_{\widetilde{f}_{0}(\lambda), \lambda}^{e}(M_{f^{\lambda}})=\sum\limits_{w=w^{(0)}w^{(1)}\in \mathcal{D}_{\widetilde{f}_{0}(\lambda)}\cap W_{\lambda}}(-1)^{\ell(w^{(1)})}q_{w^{(0)}}^{-1}q_{w^{(1)}}M_{\widetilde{f}_{0}(\lambda)}T_{w}.
\end{align*}
Note that $\mathcal{D}_{\widetilde{f}_{0}(\lambda)}\cap W_{\lambda}=\big\{s_{\lambda_{0}-1}s_{\lambda_{0}-2}\cdots s_{t}\big\}^{\lambda_{0}}_{t=1}\sqcup \big\{s_{\lambda_{0}-1}s_{\lambda_{0}-2}\cdots s_{1}s_{0}s_{1}\cdots s_{t-1}\big\}^{\lambda_{0}}_{t=1}$. Hence, we have
\begin{align*}
\phi_{\widetilde{f}_{0}(\lambda), \lambda}^{e}(M_{f^{\lambda}})=&\sum\limits_{t=1}^{\lambda_{0}}q^{-(\lambda_{0}-t)}M_{\widetilde{f}_{0}(\lambda)}T_{\lambda_{0}-1}T_{\lambda_{0}-2}\cdots T_{t}\\
&+\sum\limits_{t=1}^{\lambda_{0}}q^{-(\lambda_{0}-1+t-1)}Q_{1}^{-1}M_{\widetilde{f}_{0}(\lambda)}T_{\lambda_{0}-1}T_{\lambda_{0}-2}\cdots T_{1}T_{0}T_{1}\cdots T_{t-1}\\
=&\sum\limits_{t=1}^{\lambda_{0}}q^{-(\lambda_{0}-t)}M_{f^{\lambda}+e_{t}}+\sum\limits_{t=1}^{\lambda_{0}}q^{-\lambda_{0}-t+2}Q_{1}^{-1}M_{f^{\lambda}-e_{t}}.
\end{align*}
Recall $\varsigma_0=qQ_1^{-1}$. Thus, we obtain
\begin{align*}
\Psi(B_{0})=\sum\limits_{\lambda\in \Lambda_{m| n;d}}q^{\lambda_{1}}\phi_{\widetilde{f}_{0}(\lambda), \lambda}^{e}.
\end{align*}

When $i=r+1=m+n+1$, we have $d_{r+1}=d_{r+2}=-1$ and $|\alpha_{r+1}|=\overline{0}$. By \eqref{bimf}, we have
\begin{align*}
\Psi(B_{r+1})M_{f^{\lambda}}=&\sum\limits_{1\leq a\leq d\atop f^{\lambda}_{a}\equiv r+1}\epsilon_{r+1, a}(f^{\lambda})\Big(\prod_{s> a}\kappa_{r+1}(f^{\lambda}_{s})^{-1}\Big)M_{f^{\lambda}+e_{a}}\\
&\notag+\varsigma_{r+1}\sum\limits_{1\leq a\leq d\atop f^{\lambda}_{a}\equiv r+1}\epsilon_{r+1, a}(f^{\lambda})\Big(\prod_{s< a}\kappa_{r+1}(f^{\lambda}_{s})^{-1}\kappa_{r}(f^{\lambda}_{s})\Big)\kappa_{r+1}(f^{\lambda}_{a})^{-1}\Big(\prod_{s> a}\kappa_{r+1}(f^{\lambda}_{s})^{-1}\Big)M_{f^{\lambda}-e_{a}}\\
=&\sum\limits_{t=1}^{\lambda_{r+1}}q^{\lambda_{r+1}-t}M_{f^{\lambda}+e_{\widetilde{\lambda}_{r}+t}}+\varsigma_{r+1}\sum\limits_{t=1}^{\lambda_{r+1}}
q^{-\lambda_{r}+\lambda_{r+1}+(t-1)}M_{f^{\lambda}-e_{\widetilde{\lambda}_{r}+t}}.
\end{align*}
On the other hand, we have
\begin{align*}
\phi_{\widetilde{e}_{r}(\lambda), \lambda}^{e}(M_{f^{\lambda}})=\sum\limits_{w=w^{(0)}w^{(1)}\in \mathcal{D}_{\widetilde{e}_{r}(\lambda)}\cap W_{\lambda}}(-1)^{\ell(w^{(1)})}q_{w^{(0)}}^{-1}q_{w^{(1)}}M_{\widetilde{e}_{r}(\lambda)}T_{w}.
\end{align*}
Note that $\mathcal{D}_{\widetilde{e}_{r}(\lambda)}\cap W_{\lambda}=\big\{s_{\widetilde{\lambda}_{r}+1}s_{\widetilde{\lambda}_{r}+2}\cdots s_{\widetilde{\lambda}_{r}+t-1}\big\}^{\lambda_{r+1}}_{t=1}\sqcup \big\{s_{\widetilde{\lambda}_{r}+1}\cdots s_{d-1}s_{d}s_{d-1}\cdots s_{\widetilde{\lambda}_{r}+t}\big\}^{\lambda_{r+1}}_{t=1}$. Hence, we have
\begin{align*}
\phi_{\widetilde{e}_{r}(\lambda), \lambda}^{e}(M_{f^{\lambda}})=&\sum\limits_{t=1}^{\lambda_{r+1}}(-1)^{t-1}q^{t-1}M_{\widetilde{e}_{r}(\lambda)}T_{\widetilde{\lambda}_{r}+1}T_{\widetilde{\lambda}_{r}+2}\cdots T_{\widetilde{\lambda}_{r}+t-1}\\
&+\sum\limits_{t=1}^{\lambda_{r+1}}(-1)^{2\lambda_{r+1}-t}q^{2\lambda_{r+1}-t-1}Q_{1}M_{\widetilde{e}_{r}(\lambda)}T_{\widetilde{\lambda}_{r}+1}\cdots T_{d-1}T_{d}T_{d-1}\cdots T_{\widetilde{\lambda}_{r}+t}.
\end{align*}
We have $M_{\widetilde{e}_{r}(\lambda)}T_{\widetilde{\lambda}_{r}+1}T_{\widetilde{\lambda}_{r}+2}\cdots T_{\widetilde{\lambda}_{r}+t-1}=(-1)^{t-1}M_{f^{\lambda}-e_{\widetilde{\lambda}_{r}+t}}$. Moreover,

\begin{align*}
M_{\widetilde{e}_{r}(\lambda)}T_{\widetilde{\lambda}_{r}+1}\cdots T_{d-1}T_{d}T_{d-1}\cdots T_{\widetilde{\lambda}_{r}+t}=A+B+C+D,
\end{align*}
where
\begin{align*}
A=&Q_{1}^{-1}M_{\widetilde{e}_{r}(\lambda)}T_{\widetilde{\lambda}_{r}}\cdots T_{1}T_{0}X_{1}^{-1}T_{1}^{-1}\cdots T_{\widetilde{\lambda}_{r}+t-1}^{-1}\\
=&Q_{1}^{-1}(-q)^{\lambda_{r}}(-1)^{\lambda_{m+1}+\cdots+\lambda_{r-1}}(-1)(-1)^{\lambda_{m+1}+\cdots+\lambda_{r-1}+\lambda_{r}+(t-1)}M_{f^{\lambda}+e_{\widetilde{\lambda}_{r}+t}}\\
=&Q_{1}^{-1}q^{\lambda_{r}}(-1)^{t}M_{f^{\lambda}+e_{\widetilde{\lambda}_{r}+t}},
\end{align*}
\begin{align*}
B=&Q_{1}^{-1}(q^{-1}-q)\sum\limits_{l=1}^{\lambda_{r+1}-1}M_{\widetilde{e}_{r}(\lambda)}T_{\widetilde{\lambda}_{r}+1}\cdots T_{\widetilde{\lambda}_{r}+l-1}T_{\widetilde{\lambda}_{r}+l}T_{\widetilde{\lambda}_{r}+l-1}\cdots T_{1}T_{0}X_{1}^{-1}T_{1}^{-1}\cdots T_{\widetilde{\lambda}_{r}+t-1}^{-1}\\
=&Q_{1}^{-1}(q^{-1}-q)\sum\limits_{l=1}^{\lambda_{r+1}-1}(-1)^{l}M_{f^{\lambda}-e_{\widetilde{\lambda}_{r}+l+1}}T_{\widetilde{\lambda}_{r}+l-1}\cdots T_{1}T_{0}X_{1}^{-1}T_{1}^{-1}\cdots T_{\widetilde{\lambda}_{r}+t-1}^{-1}\\
=&Q_{1}^{-1}(q^{-1}-q)\sum\limits_{l=1}^{\lambda_{r+1}-1}(-1)^{l}(-q)^{l-1}(-1)^{\lambda_{m+1}+\cdots+\lambda_{r}}\\
&\times M_{(r+1,0^{\lambda_{0}},\ldots,r^{\lambda_{r}},(r+1)^{l-1},r,(r+1)^{\lambda_{r+1}-l-1})}T_{0}X_{1}^{-1}T_{1}^{-1}\cdots T_{\widetilde{\lambda}_{r}+t-1}^{-1}\\
=&Q_{1}^{-1}(q^{-1}-q)\sum\limits_{l=1}^{\lambda_{r+1}-1}(-1)^{l}(-q)^{l-1}(-1)^{\lambda_{m+1}+\cdots+\lambda_{r}}(-1)\\
&\times M_{(r+1,0^{\lambda_{0}},\ldots,r^{\lambda_{r}},(r+1)^{l-1},r,(r+1)^{\lambda_{r+1}-l-1})}T_{1}^{-1}\cdots T_{\widetilde{\lambda}_{r}+t-1}^{-1}\\
=&Q_{1}^{-1}(q^{-1}-q)\bigg\{\sum\limits_{l=1}^{t-1}(-1)^{l}(-q)^{l-1}(-1)(-q)^{-(t-2)}(-1)M_{(0^{\lambda_{0}},\ldots,r^{\lambda_{r}},(r+1)^{l-1},r,(r+1)^{\lambda_{r+1}-l})}\\
&+\sum\limits_{l=t}^{\lambda_{r+1}-1}(-1)^{l}(-q)^{l-1}(-1)(-q)^{-(t-1)}M_{(0^{\lambda_{0}},\ldots,r^{\lambda_{r}},(r+1)^{l},r,(r+1)^{\lambda_{r+1}-l-1})}\bigg\},
\end{align*}
\begin{align*}
C=(Q_{1}^{-1}-Q_{0}^{-1})M_{\widetilde{e}_{r}(\lambda)}T_{\widetilde{\lambda}_{r}+1}T_{\widetilde{\lambda}_{r}+2}\cdots T_{\widetilde{\lambda}_{r}+t-1}=(Q_{1}^{-1}-Q_{0}^{-1})(-1)^{t-1}M_{f^{\lambda}-e_{\widetilde{\lambda}_{r}+t}},
\end{align*}
\begin{align*}
D=&(Q_{1}^{-1}-Q_{0}^{-1})(q^{-1}-q)\sum\limits_{l=t}^{\lambda_{r+1}-1}M_{\widetilde{e}_{r}(\lambda)}T_{\widetilde{\lambda}_{r}+1}\cdots T_{\widetilde{\lambda}_{r}+l-1}T_{\widetilde{\lambda}_{r}+l}T_{\widetilde{\lambda}_{r}+l-1}\cdots T_{\widetilde{\lambda}_{r}+t}\\
=&(Q_{1}^{-1}-Q_{0}^{-1})(q^{-1}-q)\sum\limits_{l=t}^{\lambda_{r+1}-1}(-1)^{l}M_{f^{\lambda}-e_{\widetilde{\lambda}_{r}+l+1}}T_{\widetilde{\lambda}_{r}+l-1}\cdots T_{\widetilde{\lambda}_{r}+t}\\
=&(Q_{1}^{-1}-Q_{0}^{-1})(q^{-1}-q)\sum\limits_{l=t}^{\lambda_{r+1}-1}(-1)^{l}(-q)^{l-t}M_{(0^{\lambda_{0}},\ldots,r^{\lambda_{r}},(r+1)^{l},r,(r+1)^{\lambda_{r+1}-l-1})}.
\end{align*}
Thus, we have
\begin{align*}
\phi_{\widetilde{e}_{r}(\lambda), \lambda}^{e}(M_{f^{\lambda}})=&\sum\limits_{t=1}^{\lambda_{r+1}}q^{t-1}M_{f^{\lambda}-e_{\widetilde{\lambda}_{r}+t}}
+\sum\limits_{t=1}^{\lambda_{r+1}}q^{\lambda_{r}+2\lambda_{r+1}-t-1}M_{f^{\lambda}+e_{\widetilde{\lambda}_{r}+t}}\\
+\sum\limits_{t=1}^{\lambda_{r+1}}&q^{2\lambda_{r+1}-t-1}(q-q^{-1})\bigg\{\sum\limits_{l=1}^{t-1}q^{l-t+1}M_{(0^{\lambda_{0}},\ldots,r^{\lambda_{r}},(r+1)^{l-1},r,(r+1)^{\lambda_{r+1}-l})}\\
+\sum\limits_{l=t}^{\lambda_{r+1}-1}q^{l-t}&M_{(0^{\lambda_{0}},\ldots,r^{\lambda_{r}},(r+1)^{l},r,(r+1)^{\lambda_{r+1}-l-1})}\bigg\}
+\sum\limits_{t=1}^{\lambda_{r+1}}q^{2\lambda_{r+1}-t-1}(Q_{0}^{-1}Q_{1}-1)M_{f^{\lambda}-e_{\widetilde{\lambda}_{r}+t}}\\
+\sum\limits_{t=1}^{\lambda_{r+1}}q^{2\lambda_{r+1}-t-1}&(1-Q_{0}^{-1}Q_{1})(q^{-1}-q)\sum\limits_{l=t}^{\lambda_{r+1}-1}q^{l-t}M_{(0^{\lambda_{0}},\ldots,r^{\lambda_{r}},(r+1)^{l},r,(r+1)^{\lambda_{r+1}-l-1})}.
\end{align*}
In the above expression, the coefficient of $M_{f^{\lambda}+e_{\widetilde{\lambda}_{r}+t}}$ is $q^{\lambda_{r}+2\lambda_{r+1}-t-1}$, while the coefficient of $M_{f^{\lambda}-e_{\widetilde{\lambda}_{r}+t}}$ is
\begin{align*}
q^{t-1}+&\sum\limits_{m=t+1}^{\lambda_{r+1}}q^{2\lambda_{r+1}+t-2m}(q-q^{-1})+\sum\limits_{m=1}^{t-1}q^{2\lambda_{r+1}+t-2m-2}(q-q^{-1})\\
&+q^{2\lambda_{r+1}-t-1}(Q_{0}^{-1}Q_{1}-1)+\sum\limits_{m=1}^{t-1}q^{2\lambda_{r+1}+t-2m-2}(1-Q_{0}^{-1}Q_{1})(q^{-1}-q)\\
=&q^{2\lambda_{r+1}+t-3}Q_{0}^{-1}Q_{1}.
\end{align*}
Recall $\varsigma_{r+1} =q^{-1}Q_0^{-1}Q_1$. Hence we have
\begin{align*}
\Psi(B_{r+1})=\sum\limits_{\lambda\in \Lambda_{m| n;d}}q^{-\lambda_{r}-\lambda_{r+1}+1}\phi_{\widetilde{e}_{r}(\lambda), \lambda}^{e}.
\end{align*}
\end{proof}

\begin{lem}\label{phismn}
For each $\lambda\in \Lambda_{m| n;d}$, the element $\phi_{\lambda, \lambda}^{e}\in \mathbb{S}_{m| n;d}^{\jmath}$ lies in the subalgebra of $\mathbb{S}_{m| n;d}^{\jmath}$ generated by $\Psi(\mathrm{k}_{j}^{\pm 1})$ for $0\leq j\leq m+n$.
\end{lem}
\begin{proof}
For $0\leq j\leq m+n$, we have $\Psi(\mathrm{k}_{j})M_{f^{\lambda}}= \Psi(\mathrm{h}_{j}\mathrm{h}_{j+1}^{-1})M_{f^{\lambda}}$. By \eqref{hamf}, we have
$\Psi(\mathrm{k}_{j})M_{f^{\lambda}}=q^{b_{j}(\lambda)}M_{f^{\lambda}}$, where
\begin{align*}
b_{j}(\lambda)=\left\{
\begin{array}{ll}
2\lambda_{0}-\lambda_{1}&\hspace{0.35cm}\text{if $j=0$},\\[0.3em]
\lambda_{j}-\lambda_{j+1}&\hspace{0.35cm}\text{if $1\leq j< m$},\\[0.3em]
\lambda_{m}+\lambda_{m+1}&\hspace{0.35cm}\text{if $j=m$},\\[0.3em]
-\lambda_{j}+\lambda_{j+1}&\hspace{0.35cm}\text{if $m+1\leq j< m+n$},\\[0.3em]
-\lambda_{m+n}+2\lambda_{m+n+1}&\hspace{0.35cm}\text{if $j=m+n$}.
\end{array}
\right.
\end{align*}
Recall that $m\neq n$. Then it is easy to verify that $(b_{0}(\lambda),\ldots,b_{m+n}(\lambda))=(b_{0}(\mu),\ldots,b_{m+n}(\mu))$ if and only if $\lambda=\mu$.

Now fix $\lambda$. For each $\mu\neq \lambda$, we may choose $a(\mu)$ such that $b_{a(\mu)}(\lambda)\neq b_{a(\mu)}(\mu)$. We define the following element of the subalgebra:
\begin{align*}
P_{\lambda}=\prod_{\mu\in \Lambda_{m| n;d}\atop \mu\neq\lambda}\frac{\Psi(\mathrm{k}_{a(\mu)})-q^{b_{a(\mu)}(\mu)}}{q^{b_{a(\mu)}(\lambda)}-q^{b_{a(\mu)}(\mu)}}.
\end{align*}
Then $P_{\lambda}$ is well-defined and satisfies
\begin{align*}
P_{\lambda}M_{f^{\nu}}=\left\{
\begin{array}{ll}
M_{f^{\lambda}}&\hspace{0.35cm}\text{if $\nu=\lambda$},\\[0.3em]
\{0\}&\hspace{0.35cm}\text{if $\nu\neq \lambda$}.
\end{array}
\right.
\end{align*}
Since $P_{\lambda}\in \mathbb{S}_{m| n;d}^{\jmath}$, it commutes with the action of each $h\in \mathbb{H}$ on $\mathbb{V}^{\otimes d}$. That is,
\begin{align*}
P_{\lambda}(M_{f^{\nu}}\cdot h)=(P_{\lambda}M_{f^{\nu}})\cdot h=\delta_{\lambda, \nu}M_{f^{\lambda}}\cdot h.
\end{align*}
But this is precisely the action of $\phi_{\lambda, \lambda}^{e}$. Hence, we have $\phi_{\lambda, \lambda}^{e}=P_{\lambda}$.
\end{proof}

\begin{prop}\label{surjve}
Assume that $m+n\geq d$. The affine $\jmath$-Schur superalgebra $\mathbb{S}_{m| n;d}^{\jmath}$ is generated by $\Psi(B_{i})$ for $i\in \hat{I}$ and $\Psi(\mathrm{k}_{j}^{\pm 1})$ for $0\leq j\leq m+n$.
\end{prop}
\begin{proof}
Let $S$ denote the subalgebra of $\mathbb{S}_{m| n;d}^{\jmath}$ generated by $\Psi(B_{i})$ and $\Psi(\mathrm{k}_{j}^{\pm 1})$. By Lemmas \ref{bibmn} and \ref{phismn}, for each $\lambda\in \Lambda_{m| n;d}$ and $0\leq i\leq m+n$, we see that $\phi_{\widetilde{e}_{i}(\lambda), \lambda}^{e}\in S$ and $\phi_{\widetilde{f}_{i}(\lambda), \lambda}^{e}\in S$.

Since $m+n\geq d$, we can choose
\begin{align*}
\omega=\{0,\underbrace{1,\ldots,1}_{d},\underbrace{0,\ldots,0}_{m+n-d},0\}\in \Lambda_{m| n;d}.
\end{align*}
For any $\lambda\in \Lambda_{m| n;d}$, it is easy to verify that there exists a sequence $(x_{1},\ldots, x_{l})$ of $\widetilde{e}_{i}, \widetilde{f}_{i}$'s such that $\lambda=x_{l}\cdots x_{1}(\omega)$ and $\{e\}=W_{\lambda_{0}}\subseteq W_{\lambda_{1}}\subseteq \cdots \subseteq W_{\lambda_{l}}=W_{\lambda}$, where $\lambda_{0}=\omega$ and $\lambda_{i}=x_{i}(\lambda_{i-1})$. Thus, we have
\begin{align*}
\phi_{\lambda, \omega}^{e}=\phi_{\lambda_{l}, \lambda_{l-1}}^{e}\cdots \phi_{\lambda_{1}, \lambda_{0}}^{e}\in S.
\end{align*}
Similarly, we obtain $\phi_{\omega, \lambda}^{e}\in S$.

By a direct calculation, we have
\begin{align*}
\phi_{\omega, \widetilde{e}_{i}(\omega)}^{e}\cdot \phi_{\widetilde{e}_{i}(\omega), \omega}^{e}=\phi_{\omega, \omega}^{e}+q_{s_{i}}^{-1}\phi_{\omega, \omega}^{s_{i}}\qquad \text{for $0\leq i\leq \mathrm{min}\{m, d-1\}$},
\end{align*}
\begin{align*}
\phi_{\omega, \widetilde{e}_{i}(\omega)}^{e}\cdot \phi_{\widetilde{e}_{i}(\omega), \omega}^{e}=\phi_{\omega, \omega}^{e}+(-q_{s_{i}})\phi_{\omega, \omega}^{s_{i}}\qquad \text{for $d> m+1$ and $m+1\leq i\leq d-1$},
\end{align*}
\begin{align*}
\phi_{\omega, \widetilde{f}_{r}(\omega)\cdots\widetilde{f}_{d}(\omega)}^{e}\cdot \phi_{\widetilde{f}_{r}(\omega)\cdots\widetilde{f}_{d}(\omega), \omega}^{e}=\phi_{\omega, \omega}^{e}+(-q_{s_{d}})\phi_{\omega, \omega}^{s_{d}}.
\end{align*}
Note that $\phi_{\widetilde{f}_{r}(\omega)\cdots\widetilde{f}_{d}(\omega), \omega}^{e}=\phi_{\widetilde{f}_{r}(\omega)\cdots\widetilde{f}_{d}(\omega), \widetilde{f}_{r-1}(\omega)\cdots\widetilde{f}_{d}(\omega)}^{e}\cdots \phi_{\widetilde{f}_{d}(\omega), \omega}^{e}\in S$. Similarly, we have $\phi_{\omega, \widetilde{f}_{r}(\omega)\cdots\widetilde{f}_{d}(\omega)}^{e}\in S$. Summarizing the above, we have shown that $\phi_{\omega, \omega}^{s_{i}}\in S$ for $0\leq i\leq d$. Since $\phi_{\omega, \omega}^{e}\mathbb{S}_{m| n;d}^{\jmath}\phi_{\omega, \omega}^{e}$ is generated by $\phi_{\omega, \omega}^{s_{i}}$ for $0\leq i\leq d$, we have $\phi_{\omega, \omega}^{e}\mathbb{S}_{m| n;d}^{\jmath}\phi_{\omega, \omega}^{e}\subseteq S$.

Finally, since
\begin{align*}
\mathbb{S}_{m| n;d}^{\jmath}=\bigoplus_{\lambda, \mu}\phi_{\lambda, \omega}^{e}\mathbb{S}_{m| n;d}^{\jmath}\phi_{\omega, \mu}^{e},
\end{align*}
and for any $\lambda, \mu\in \Lambda_{m| n;d}$, we have
\begin{align*}
\phi_{\lambda, \omega}^{e}\mathbb{S}_{m| n;d}^{\jmath}\phi_{\omega, \mu}^{e}=\phi_{\lambda, \omega}^{e}\phi_{\omega, \omega}^{e}\mathbb{S}_{m| n;d}^{\jmath}\phi_{\omega, \omega}^{e}\phi_{\omega, \mu}^{e}\subseteq S,
\end{align*}
we see that $S=\mathbb{S}_{m| n;d}^{\jmath}$.
\end{proof}

\begin{thm}\label{schruimo}
Suppose that $m+n\geq d\geq 2$. We have the following isomorphisms:
\begin{align*}
\Psi(\widehat{U}^{\jmath}(\mathfrak{sl}))\cong \mathrm{End}_{\mathbb{H}}(\mathbb{V}^{\otimes d}),
\end{align*}
\begin{align*}
\mathrm{End}_{\widehat{U}^{\jmath}(\mathfrak{sl})}(\mathbb{V}^{\otimes d})\cong \mathbb{H}^{op}.
\end{align*}
\end{thm}
\begin{proof}
By Prop. \ref{surjve} we have $\Psi(\widehat{U}^{\jmath}(\mathfrak{sl}))=\mathbb{S}_{m| n;d}^{\jmath}$. Hence the first isomorphism follows from Lemma \ref{mathtmnd}.

Note that $W_{\omega}=\{e\}$, and consequently, $x_{\omega^{(0)}}y_{\omega^{(1)}}\mathbb{H}=\mathbb{H}$. Thus, $\mathbb{T}_{m| n;d}$ contains the right regular $\mathbb{H}$-module $\mathbb{H}_{\mathbb{H}}$ as a direct summand. Hence, we have $\mathbb{S}_{m| n;d}^{\jmath}\phi_{\omega, \omega}^{e}\cong \mathbb{T}_{m| n;d}$, and
\begin{align*}
\mathrm{End}_{\mathbb{S}_{m| n;d}^{\jmath}}(\mathbb{T}_{m| n;d})\cong \mathrm{End}_{\mathbb{S}_{m| n;d}^{\jmath}}(\mathbb{S}_{m| n;d}^{\jmath}\phi_{\omega, \omega}^{e})\cong (\phi_{\omega, \omega}^{e}\mathbb{S}_{m| n;d}^{\jmath}\phi_{\omega, \omega}^{e})^{op}\cong (\mathrm{End}_{\mathbb{H}}(\mathbb{H}_{\mathbb{H}}))^{op}\cong \mathbb{H}^{op}.
\end{align*}
Since $\Psi(\widehat{U}^{\jmath}(\mathfrak{sl}))=\mathbb{S}_{m| n;d}^{\jmath}$, we have the isomorphism
\begin{align*}
\mathrm{End}_{\widehat{U}^{\jmath}(\mathfrak{sl})}(\mathbb{V}^{\otimes d})\cong \mathrm{End}_{\kappa\mathbb{S}_{m| n;d}^{\jmath}\kappa^{-1}}(\mathbb{V}^{\otimes d})\cong
\mathrm{End}_{\mathbb{S}_{m| n;d}^{\jmath}}(\mathbb{T}_{m| n;d}).
\end{align*}
Thus, we obtain the second isomorphism.
\end{proof}

\section{The $\imath$-quantum affine superalgebras}\label{def:quanafjsi}
In this section, we introduce a variant of $\jmath$-quantum affine superalgebras associated to odd-rank quantum affine superalgebras, and establish the second Schur--Weyl reciprocity. We only state the main results and omit their proofs, since they can be proved by arguments similar to those in \S\S\ref{def:quanaffis}-\ref{swrecio}.

\subsection{}\label{iquma}
Set $\mathfrak{n} :=N-1=2r+1$ for $r\in \mathbb{Z}_{\geq 2}$. We consider the general linear Lie superalgebra $\mathfrak{g}$, whose corresponding root system has rank $\mathfrak{n}-1$, and the associated quantum affine superalgebra $\widehat{U}_{\mathfrak{n}}$ with generators $E_i, F_i$ for $i\in \hat{I}_{\mathfrak{n}}$ and $q^{\pm \epsilon_{1}^\vee},\ldots, q^{\pm \epsilon_{\mathfrak{n}}^\vee}$, where $\hat{I}_{\mathfrak{n}}=\left\{ 0,1,\ldots, \mathfrak{n}-1 \right\}$.

Assume that $|\alpha_{0}|=\overline{0}$, and there exists a parity-preserving permutation $\tau$ of $\hat{I}_{\mathfrak{n}}$ such that $\tau(0)=0$, $\tau^{2}=\mathrm{id}$ and $\tau(i)=\mathfrak{n}-i$ for $1\leq i\leq r$. Given parameters $\kappa_{0}, \varsigma_{i}\in \F$, we set $B_{0} :=F_{0}+\varsigma_{0}E_{0}K_{0}^{-1}+\kappa_{0}K_{0}^{-1}$ and $B_{i} :=F_{i}+\varsigma_{i}E_{\tau i}K_{i}^{-1}$ for $i\in \hat{I}_{\mathfrak{n}}\setminus\{0\}$. We require that $\varsigma_{i}=\varsigma_{\tau i}$ for $1\leq i\leq r-1$ and $d_{i}=(-1)^{|\alpha_i|}d_{\tau i}$ for $1\leq i\leq r$.
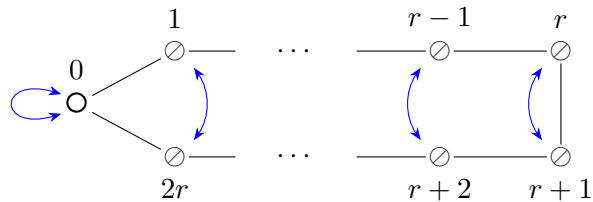
\begin{figure}[htbp]
\centering
\begin{tikzpicture}[
    scale=1,
    vertex/.style={
        inner sep=0pt,
        outer sep=1pt
    },
    >=Stealth
]

\node[vertex,label=above:$0$]
    (a0) at (0,0) {\scalebox{1.8}{$\circ$}};

\node[vertex,label=above:$1$]
    (a1) at (1.3,0.7) {$\oslash$};

\node[vertex,label=below:$2r$]
    (a2r) at (1.3,-0.7) {$\oslash$};

\node[vertex,label=above:$r-1$]
    (am) at (4.8,0.7) {$\oslash$};

\node[vertex,label=below:$r+2$]
    (bm) at (4.8,-0.7) {$\oslash$};

\node[vertex,label=above:$r$]
    (ar) at (6.4,0.7) {$\oslash$};

\node[vertex,label=below:$r+1$]
    (br) at (6.4,-0.7) {$\oslash$};

\draw (a0) -- (a1);
\draw (a0) -- (a2r);

\draw (a1) -- (2.1,0.7);
\draw (a2r) -- (2.1,-0.7);

\node at (2.9,0.7) {$\cdots$};
\node at (2.9,-0.7) {$\cdots$};

\draw (3.7,0.7) -- (am);
\draw (3.7,-0.7) -- (bm);

\draw (am) -- (ar);
\draw (bm) -- (br);
\draw (ar) -- (br);

\draw[blue,<->,bend left=35]
    (1.55,0.48) to (1.55,-0.48);

\draw[blue,<->,bend right=35]
    (4.55,0.48) to (4.55,-0.48);

\draw[blue,<->,bend right=35]
    (6.15,0.48) to (6.15,-0.48);

\draw[blue,<->]
    (-0.20,0.13)
    .. controls (-1.05,0.38)
       and (-1.05,-0.38)
    .. (-0.20,-0.13);

\end{tikzpicture}
\end{figure}


Let $\widehat{U}(\mathfrak{sl}_{\mathfrak{n}})$ be the subalgebra of $\widehat{U}_{\mathfrak{n}}$ generated by $E_i$, $F_i$, $K_i^{\pm 1}$ for $i\in \hat{I}_{\mathfrak{n}}$, where $K_i=q^{d_{i}\epsilon_{i}^\vee-d_{i+1}\epsilon_{i+1}^\vee}$ for $1\leq i\leq \mathfrak{n}-1$ and $K_{0}=q^{d_{\mathfrak{n}}\epsilon_{\mathfrak{n}}^\vee-d_{1}\epsilon_{1}^\vee}$. Let $\widehat{U}^{\imath}(\mathfrak{sl}_{\mathfrak{n}})$ be the subalgebra of $\widehat{U}(\mathfrak{sl}_{\mathfrak{n}})$ generated by $B_{i}$ for $i\in \hat{I}_{\mathfrak{n}}$ and $\mathrm{k}_{j}^{\pm 1}$ for $1\leq j\leq r$, where $\mathrm{k}_{j}=K_{j}K_{\tau j}^{-1}$. Then we have the following theorem analogous to Theorem \ref{def:ipreseabslma}.

\begin{thm}
\label{def:ipreseabslmadada}
$\widehat{U}^{\imath}(\mathfrak{sl}_{\mathfrak{n}})$ is a right coideal subalgebra of $\widehat{U}(\mathfrak{sl}_{\mathfrak{n}})$. Moreover, it admits the following presentation: for $1\leq a, b\leq r$ with $\mathrm{k}_{\tau a}= \mathrm{k}_{a}^{-1}$ and $i\in \hat{I}_{\mathfrak{n}}$,

\begin{align}\label{KEF1relada}
    \mathrm{k}_{a}\mathrm{k}_{b}=\mathrm{k}_{b}\mathrm{k}_{a},\qquad
    \mathrm{k}_{a}\mathrm{k}_{a}^{-1}=1=\mathrm{k}_{a}^{-1}\mathrm{k}_{a},
\end{align}
	
\begin{align}\label{skunk2relaad}
    \mathrm{k}_{a}B_{i}=q^{-(\alpha_{a}, \alpha_i)+(\alpha_{\tau a}, \alpha_i)}B_{i}\mathrm{k}_{a},
\end{align}
\begin{align}\label{square33adad}
    \text{for }1\leq i\leq r-1, \qquad [B_{\tau i},B_i]=
\frac{\varsigma_{\tau i}\mathrm{k}_{i} - \varsigma_{i}\mathrm{k}_{i}^{-1}}{q_i-q_i^{-1}},
\end{align}

\begin{align}\label{square3333adad}
    \text{for $i\not\equiv j\pm 1$ (mod $\mathfrak{n}$) and $j\neq \tau i$}, \qquad [B_{j},B_i]=0,
\end{align}

\begin{align*}
    \text{for }|\alpha_i|=\overline{0}, i\neq j \text{ and }i\sim j, \text{ we ~have }\\
    \end{align*}
\begin{align}\label{squareSerre5adad}
    [B_i,[B_i,B_j]_{q_i}]_{q_i^{-1}}=\begin{cases}
q_i\varsigma_{i}B_{j}&\hspace{0.35cm}\text{if }i=\tau i,\\[0.3em]
-(q_i+q_i^{-1})(q_i\varsigma_{\tau i}\mathrm{k}_{i}B_i+q_i^{-2}\varsigma_{i}\mathrm{k}_{i}^{-1}B_i)&\hspace{0.35cm}\text{if }j=\tau i,\\[0.3em]
0&\hspace{0.35cm}\text{otherwise},
\end{cases}
\end{align}
\begin{align*}
    \text{for }|\alpha_j|=\overline{1}, i=j-1 \text{ and }k=j+1, \text{ we have } \\
\end{align*}
\begin{align}\label{Serre6adad}
        &[[[B_{i},B_j]_{q_j},B_{k}]_{q_{k}}],B_j]\\
        =&\notag \begin{cases}
0&\hspace{0.35cm}\text{if }\{i, j, k\}\cap \{\tau i, \tau j, \tau k\}\in \{\varnothing, \{i\}, \{k\}\},\\[0.3em]
(-1)^{|\alpha_i|}\Big(q_{\tau j}^{-1}\varsigma_{\tau j}[B_{j},B_{i}]_{q_j}\mathrm{k}_{j}+\varsigma_{j}[B_{j},B_{i}]_{q_{\tau j}}\mathrm{k}_{j}^{-1}\Big)&\hspace{0.35cm}\text{if }k=\tau j,\\[0.3em]
 (-1)^{|\alpha_k|}\Big(q_j^{-1}\varsigma_{\tau j}[B_j,B_k]_{q_{\tau j}}\mathrm{k}_j+\varsigma_j[B_j,B_k]_{q_j}\mathrm{k}_{j}^{-1}\Big)&\hspace{0.35cm}\text{if }i=\tau j,
\end{cases}
\end{align}
where all indices are taken modulo $\mathfrak{n}$.
\end{thm}

\subsection{}
Recall that in \S\ref{comm} we have defined an infinite dimensional superspace $\mathbb{V}$ with an $\F$-basis $\{v_{j}\:|\:j\in \mathbb{Z}\}$ and parities $|v_{j}|=p(j)$. Let $W$ be the vector superspace over $\C$ with a basis $\{v_{1}, v_{2},\ldots,v_{\mathfrak{n}}\}$ and parities $|v_{j}|=p(j)$. Associated to $W$, we have the general linear Lie superalgebra $\mathfrak{g} = \mathfrak{gl}(W)$; moreover, in $\S$\ref{iquma}, we have defined the corresponding quantum affine superalgebra and its coideal subalgebra, which are also denoted by $\widehat{U}(\mathfrak{sl}_{\mathfrak{n}})$ and $\widehat{U}^{\imath}(\mathfrak{sl}_{\mathfrak{n}})$, respectively.

Let $\mathbb{V}_{\mathfrak{n}}$ be the $\F$-subspace of $\mathbb{V}$ spanned by $\{v_{j}\:|\:j\in \mathbb{Z} \text{ and }j\not\equiv 0 \text{ (mod $N$)}\}$. Then we have the following lemma.
\begin{lem}\label{uvacinotda}
The following defines an action of $\widehat{U}(\mathfrak{sl}_{\mathfrak{n}})$ on $\mathbb{V}_{\mathfrak{n}}$:
\begin{align*}
E_{i} v_{j+1}=\left\{
\begin{array}{ll}
v_j&\hspace{0.35cm}\text{if $j\equiv i\neq 0\pmod N$},\\[0.3em]
v_{j-1}&\hspace{0.35cm}\text{if $j\equiv i= 0\pmod N$},\\[0.3em]
0&\hspace{0.35cm}\text{otherwise}.
\end{array}
\right.
\end{align*}
\begin{align*}
F_{i} v_{j}=\left\{
\begin{array}{ll}
v_{j+1}&\hspace{0.35cm}\text{if $j\equiv i\neq 0\pmod N$},\\[0.3em]
v_{j+2}&\hspace{0.35cm}\text{if $j+1\equiv i=0\pmod N$},\\[0.3em]
0&\hspace{0.35cm}\text{otherwise}.
\end{array}
\right.
\end{align*}
\begin{align*}
K_{0} v_{j}=\left\{
\begin{array}{ll}
q^{d_{\mathfrak{n}}}v_j&\hspace{0.35cm}\text{if $j\equiv -1\pmod N$},\\[0.3em]
q^{-d_{1}}v_j&\hspace{0.35cm}\text{if $j\equiv 1\pmod N$},\\[0.3em]
v_j&\hspace{0.35cm}\text{otherwise}.
\end{array}
\right.
\end{align*}
\begin{align*}
\text{For }1\leq i\leq \mathfrak{n}-1,\qquad K_{i} v_{j}=\left\{
\begin{array}{ll}
q^{d_{i}}v_j&\hspace{0.35cm}\text{if $j\equiv i\pmod N$},\\[0.3em]
q^{-d_{i+1}}v_j&\hspace{0.35cm}\text{if $j\equiv i+1\pmod N$},\\[0.3em]
v_j&\hspace{0.35cm}\text{otherwise}.
\end{array}
\right.
\end{align*}
\end{lem}

From Lemma \ref{uvacinotda}, we can obtain an explicit description of the action of $\widehat{U}^{\imath}(\mathfrak{sl}_{\mathfrak{n}})$ on $\mathbb{V}_{\mathfrak{n}}$.
\begin{lem}\label{uimahtvacinotha}
The following defines an action of $\widehat{U}^{\imath}(\mathfrak{sl}_{\mathfrak{n}})$ on $\mathbb{V}_{\mathfrak{n}}$:
\begin{align*}
\text{For } 1\le a\le r-1,\qquad
\mathrm{k}_a v_j= \begin{cases} q^{d_a}v_j, &\text{if }j\equiv a\text{ or }\mathfrak{n}-a+1\pmod N,\\[0.1em] q^{-d_{a+1}}v_j, &\text{if }j\equiv a+1\text{ or }\mathfrak{n}-a\pmod N,\\[0.1em] v_j,&\text{otherwise}. \end{cases}
\end{align*}
\begin{align*}
\mathrm{k}_{r} v_j= \begin{cases} q^{d_r}v_j, &\text{if }j\equiv r\text{ or }r+2\pmod N,\\[0.1em] q^{-2d_{r+1}}v_j, &\text{if }j\equiv r+1\pmod N,\\[0.1em] v_j,&\text{otherwise}. \end{cases}
\end{align*}
\begin{align*}
\text{For } 1\le i\le \mathfrak{n}-1\text{ and }i\ne r+1,\qquad
B_iv_j= \begin{cases} v_{j+1},&\text{if }j\equiv i\pmod N,\\[0.1em] \varsigma_i v_{j-1},&\text{if }j\equiv \mathfrak{n}-i+1\pmod N,\\[0.1em] 0,&\text{otherwise}. \end{cases}
\end{align*}
\begin{align*}
B_{r+1}v_j= \begin{cases} v_{j+1}+\varsigma_{r+1}q^{-d_{r+1}}v_{j-1}, &\text{if }j\equiv r+1\pmod N,\\[0.1em] 0,&\text{otherwise}. \end{cases}
\end{align*}
\begin{align*} B_0v_j= \begin{cases} v_{j+2}+\kappa_0q^{-d_{\mathfrak{n}}}v_j, &\text{if }j\equiv-1\pmod N,\\[0.1em] \varsigma_0q^{d_1}v_{j-2} +\kappa_0q^{d_1}v_j, &\text{if }j\equiv1\pmod N,\\[0.1em] \kappa_0v_j, &\text{otherwise}. \end{cases}
\end{align*}
\end{lem}

Using the iterated coproduct, we can define an action of $\widehat{U}(\mathfrak{sl}_{\mathfrak{n}})$ on $\mathbb{V}_{\mathfrak{n}}^{\otimes d}$ from the left as follows.
\begin{lem}\label{auivduddav}
The following defines an action of $\widehat{U}(\mathfrak{sl}_{\mathfrak{n}})$ on $\mathbb{V}_{\mathfrak{n}}^{\otimes d}$: for $i \in \hat{I}_{\mathfrak{n}}$ and $M_{f}\in \mathbb{V}_{\mathfrak{n}}^{\otimes d}$,
\begin{align*}
E_{i}M_{f}=\sum\limits_{a\in A_{i}(f)}\epsilon_{i, a}(f)\Big(\prod_{b< a}\chi_{i}(f_{b})\Big)M_{f-s_{i}e_{a}},
\end{align*}
\begin{align*}
F_{i}M_{f}=\sum\limits_{a\in D_{i}(f)}\epsilon_{i, a}(f)\Big(\prod_{b> a}\chi_{i}(f_{b})^{-1}\Big)M_{f+s_{i}e_{a}},
\end{align*}
\begin{align*}
K_{i}M_{f}=\prod_{b=1}^{d}\chi_{i}(f_{b})M_{f},
\end{align*}
where
\begin{align*}
A_{i}(f)= \begin{cases} \{a\:|\:f_a\equiv 1\pmod N\}, &\text{if }i=0,\\[0.1em] \{a\:|\:f_a\equiv i+1\pmod N\},&\text{if }i\neq 0, \end{cases}
\end{align*}
\begin{align*}
D_{i}(f)= \begin{cases} \{a\:|\:f_a\equiv -1\pmod N\}, &\text{if }i=0,\\[0.1em] \{a\:|\:f_a\equiv i\pmod N\},&\text{if }i\neq 0, \end{cases}
\end{align*}
\begin{align*}
\chi_{i}(t)= \begin{cases} q^{d_{i}\delta_{t\equiv i}-d_{i+1}\delta_{t\equiv i+1}}, &\text{if }i\neq 0,\\[0.1em] q^{d_{\mathfrak{n}}\delta_{t\equiv -1}-d_{1}\delta_{t\equiv 1}},&\text{if }i=0, \end{cases}\qquad s_{i}= \begin{cases} 2, &\text{if }i=0,\\[0.1em] 1,&\text{if }i\neq 0. \end{cases}
\end{align*}
\end{lem}

By restriction we obtain an action of $\widehat{U}^{\imath}(\mathfrak{sl}_{\mathfrak{n}})$ on $\mathbb{V}_{\mathfrak{n}}^{\otimes d}$.
\begin{lem}\label{adviuud}
The following defines an action of $\widehat{U}^{\imath}(\mathfrak{sl}_{\mathfrak{n}})$ on $\mathbb{V}_{\mathfrak{n}}^{\otimes d}$: for $1\leq a\leq r$, $i \in \hat{I}_{\mathfrak{n}}$ and $M_{f}\in \mathbb{V}_{\mathfrak{n}}^{\otimes d}$,
\begin{align}\label{hafmha}
\mathrm{k}_a M_{f}=\Big(\prod_{b=1}^{d}\chi_{a}(f_{b})\chi_{\tau a}(f_{b})^{-1}\Big)M_{f}=q^{d_a[\gamma_{a}(f)+\gamma_{\mathfrak{n}-a+1}(f)]-d_{a+1}[\gamma_{a+1}(f)+\gamma_{\mathfrak{n}-a}(f)]}M_{f},
\end{align}
and
\begin{align}\label{bifmib}
B_{i}M_{f}=&\sum\limits_{a\in D_{i}(f)}\epsilon_{i, a}(f)\Big(\prod_{b> a}\chi_{i}(f_{b})^{-1}\Big)M_{f+s_{i}e_{a}}\\
&\notag+\varsigma_{i}\sum\limits_{a\in A_{\tau i}(f)}\epsilon_{i, a}(f)\Big(\prod_{b< a}\chi_{i}(f_{b})^{-1}\chi_{\tau i}(f_{b})\Big)\chi_{i}(f_{a})^{-1}\Big(\prod_{b> a}\chi_{i}(f_{b})^{-1}\Big)M_{f-s_{i}e_{a}}\\
&\notag+\delta_{i0}\kappa_{0}\Big(\prod_{b=1}^{d}\chi_{0}(f_{b})^{-1}\Big)M_{f}.
\end{align}
\end{lem}

Recall that in Theorem \ref{the:ipreseab} we have defined a right action of $\mathbb{H}$ on $\mathbb{V}^{\otimes d}$. It is easy to verify that it naturally restricts to a right $\mathbb{H}$-module structure on $\mathbb{V}_{\mathfrak{n}}^{\otimes d}$. Then we have a variant of Theorem \ref{hatcuth}.

\begin{thm}\label{icommu}
Assume that the parameters of $\widehat{U}^{\imath}(\mathfrak{sl}_{\mathfrak{n}})$ are specialized by
$$ \kappa_0=\frac{Q_0-Q_1^{-1}}{q-q^{-1}},\qquad \varsigma_0=q^{-1}Q_0Q_1^{-1},\qquad \varsigma_{r+1} =q^{-1}Q_0^{-1}Q_1,$$
and
$$ \varsigma_i=(-1)^{|\alpha_i|} \qquad \text{for }i\in \hat{I}_{\mathfrak{n}}\setminus \{0,r+1\}.$$
Then on $\mathbb{V}_{\mathfrak{n}}^{\otimes d}$, the left action of $\widehat{U}^{\imath}(\mathfrak{sl}_{\mathfrak{n}})$, commutes with the right one of $\mathbb{H}$.
\end{thm}

From now on, we assume that $m+n\geq d$. Recall $\Lambda_{m| n;d}$ defined in \eqref{ldabma}. We define
\begin{align*}
\Lambda_{m| n;d}^{\imath} :=\Big\{\lambda=(\lambda_{0}, \lambda_{1},\ldots,\lambda_{m}|\lambda_{m+1},\ldots,\lambda_{m+n+1})\in \Lambda_{m| n;d}\big|\lambda_{0}=0\Big\},
\end{align*}
and the following right $\mathbb{H}$-module
\begin{align*}
\mathbb{T}_{m| n;d}^{\imath} :=\bigoplus_{\lambda\in \Lambda_{m| n;d}^{\imath}}x_{\lambda^{(0)}}y_{\lambda^{(1)}}\mathbb{H},
\end{align*}
whose $\Z_2$-grading is similar to that of $\mathbb{T}_{m| n;d}$. Then we have the following variant of Lemma \ref{mathtmnd}.
\begin{lem}\label{iisom}
There exists an grading-preserving isomorphism of $\mathbb{H}$-modules $\mathbb{T}_{m| n;d}^{\imath}\cong \mathbb{V}_{\mathfrak{n}}^{\otimes d}$.
\end{lem}
Using $\mathbb{T}_{m| n;d}^{\imath}$, we can define an $\imath$-variant of $\mathbb{S}_{m| n;d}^{\jmath}$, called the affine $\imath$-Schur superalgebra, as follows:
\begin{align*}
\mathbb{S}_{m| n;d}^{\imath} :=\mathrm{End}_{\mathbb{H}}(\mathbb{T}_{m| n;d}^{\imath})=\bigoplus_{\lambda, \mu\in \Lambda_{m| n;d}^{\imath}}\mathrm{Hom}_{\mathbb{H}}(x_{\mu^{(0)}}y_{\mu^{(1)}}\mathbb{H}, x_{\lambda^{(0)}}y_{\lambda^{(1)}}\mathbb{H}),
\end{align*}
whose $\Z_2$-grading is similar to that of $\mathbb{S}_{m| n;d}^{\jmath}$.

By Theorem \ref{icommu} and Lemma \ref{iisom}, we obtain an $\F$-algebra homomorphism $\Psi^{\imath} :\widehat{U}^{\imath}(\mathfrak{sl}_{\mathfrak{n}})\rightarrow \mathbb{S}_{m| n;d}^{\imath}$. Then we have an analog of Prop. \ref{surjve}.
\begin{prop}
Assume that $m+n\geq d$. The affine $\imath$-Schur superalgebra $\mathbb{S}_{m| n;d}^{\imath}$ is generated by $\Psi^{\imath}(B_{i})$ for $i\in \hat{I}_{\mathfrak{n}}$ and $\Psi^{\imath}(\mathrm{k}_{j}^{\pm 1})$ for $1\leq j\leq r$.
\end{prop}

Finally, we have a variant of Theorem \ref{schruimo}, which can be proved similarly.
\begin{thm}\label{suchuri}
Suppose that $m+n\geq d\geq 2$. We have the following isomorphisms:
\begin{align*}
\Psi^{\imath}(\widehat{U}^{\imath}(\mathfrak{sl}_{\mathfrak{n}}))\cong \mathrm{End}_{\mathbb{H}}(\mathbb{V}_{\mathfrak{n}}^{\otimes d}),
\end{align*}
\begin{align*}
\mathrm{End}_{\widehat{U}^{\imath}(\mathfrak{sl}_{\mathfrak{n}})}(\mathbb{V}_{\mathfrak{n}}^{\otimes d})\cong \mathbb{H}^{op}.
\end{align*}
\end{thm}

\begin{rem}
In \cite[\S\S4.1 and 4.3]{FLLLWW20}, the authors introduced two additional coideal subalgebras of quantum affine algebras of type $A$. Analogous constructions can be carried out in the super setting, yielding two further coideal sub-superalgebras of quantum affine superalgebras of type $A$. Since the arguments are similar, we omit the details and leave them to the reader.
\end{rem}


\end{document}